\documentclass[a4paper,11pt]{amsart}
\usepackage{amsmath, amssymb}
\usepackage{hyperref, xcolor}
\hypersetup{colorlinks=true,linkcolor=blue,citecolor=blue}

\usepackage[normalem]{ulem}
\usepackage[polish,english]{babel}
\usepackage{comment}
\usepackage{graphicx}
\usepackage{array}
\usepackage{tabularx}
\input cyracc.def

\numberwithin{equation}{section}
\newtheorem{thm}{Theorem}[section]

\newtheorem{prop}[thm]{Proposition}
\newtheorem{lem}[thm]{Lemma}
\newtheorem{cor}[thm]{Corollary}
\newtheorem{clm}[thm]{Claim}

\theoremstyle{definition}
\newtheorem{defn}[thm]{Definition}

\theoremstyle{remark}
\newtheorem{rem}[thm]{Remark}

\newcommand{\N}{\mathbb{N}}

\newcommand{\Q}{\mathbb{Q}}
\newcommand{\R}{\mathbb{R}}
\renewcommand{\P}{\mathcal{P}}
\newcommand{\E}{\mathcal{E}}
\newcommand{\uB}{{\underline{B}}}
\newcommand{\mux}{\mu_{X}}
\newcommand{\muy}{\mu_{Y}}

\newcommand{\AC}{\mathrm{AC}^2}

\newcommand{\Bu}{\mathcal B^{*}}

\newcommand{\defset}[2]{ \left\{ #1 : #2 \right\} }

\newcommand{\Entm}[2]{\ensuremath{{\rm Ent}}_{\mu_{#1}} (#2)}
\newcommand{\Ent}[2]{\ensuremath{{\rm Ent}}_{#1} (#2)}

\newcommand{\DomEnt}[1]{ D(\ensuremath{{\rm Ent}}_{\mu_{#1}}) }

\newcommand{\seqk}[1]{ \{#1\}_{k=1}^{\infty} }

\newcommand{\seqn}[1]{ \{#1\}_{n=1}^{\infty} }

\DeclareMathOperator{\supp}{supp}

\newcommand{\niN}{{n\in\N}}
\newcommand{\iiI}{{i\in I}}
\newcommand{\kti}{{k\to\infty}}

\newcommand{\limn}{\lim_{n\to\infty}}

\newcommand{\limsn}{\limsup_{n\to\infty}}

\newcommand{\limin}{\liminf_{n\to\infty}}

\newcommand{\ep}{\varepsilon}
\renewcommand{\phi}{\varphi}

\newcommand{\cB}{\mathcal{B}}

\newcommand{\lip}{\mathcal{L}ip}

\newcommand{\Lip}{{\rm Lip}}

\newcommand{\fstop}{\,.}
\newcommand{\comma}{\, ,}
\newcommand{\Ch}{{\sf Ch}}
\newcommand{\CD}{{\sf CD}}
\newcommand{\RCD}{{\sf RCD}}
\newcommand{\EVI}{{\sf EVI}}

\newcommand{\LS}{{\sf LS}}
\newcommand{\PI}{{\sf PI}}
\newcommand{\LH}{{\sf LH}}
\newcommand{\DFH}{{\sf H}_\infty}

\newcommand{\TL}{{\sf T}}

\newcommand{\ENT}{{\rm Ent}}

\renewcommand{\c}{\mathrm{c}}

\newcommand{\diff}{\mathop{}\!\mathrm{d}}	
\newcommand{\Pac}{\mathcal P_*}
\newcommand{\F}{\mathcal F}
\newcommand{\e}{\varepsilon}

\newcommand{\SL}[2]{\ensuremath{{D}}_{#1} (#2)}

\usepackage[usenames,dvipsnames]{xcolor}
\newcommand{\purple}{\textcolor{purple}}
\newcommand{\blue}{\textcolor{blue}}
\newcommand{\ilim}{\varprojlim}
\newcommand{\dlim}{\varinjlim}
\newcommand{\Cpl}{{\rm Cpl}}

\title[Curvature Bounds on Inverse Limit]{Stability of Synthetic Ricci Curvature Lower Bounds for Inverse Limit Extended Metric Measure Spaces}
\author[K.Suzuki]{Kohei Suzuki}
\address{Kohei Suzuki \\ Department of Mathematical Science, Durham University, Science Laboratories, South Road, DH1 3LE, United Kingdom}
\email{kohei.suzuki@durham.ac.uk}
\author[T. Yokota]{Takumi Yokota}
\address{Takumi Yokota \\ Mathematical Institute, Tohoku University, Sendai 980-8578, JAPAN}
\email{takumiy@tohoku.ac.jp}
\thanks{K.S.~gratefully acknowledges funding by the Alexander von Humboldt Stiftung.
T.Y. was partly supported by JSPS KAKENHI (No.18K03298).}
\date{\today}
\begin{document}
\maketitle

\begin{abstract}
We show that every Polish extended metric measure space arises as an
inverse limit of metric measure spaces up to isomorphism. We then prove that synthetic Ricci curvature lower bounds and several functional inequalities, including the
log-Sobolev, Talagrand, Poincar\'e, and dimension-free Harnack inequalities are stable under inverse limit. 
We discuss applications to infinite-dimensional spaces, including
abstract Wiener spaces and their quotient spaces.
\end{abstract}
\section{Introduction} 
\paragraph{\bf Inverse limit spaces}
In this paper, we say that $(X, d_X, \mu_X)$ is a metric measure space (or mm-space) if $(X, d_{X})$ is a complete separable metric space and $\mu_{X}$ is a fully supported Borel probability measure with the Borel $\sigma$-algebra in terms of the topology $\tau_{d_X}$ induced by $d_X$. 
An {\it inverse system} $\seqn{(X_n, d_{X_n}, \mu_{X_n}), f_n}$  is a sequence of metric measure spaces $(X_n, d_{X_n}, \mu_{X_n})$ with $1$-Lipschitz maps $f_{n}: X_{n+1} \to X_n$ satisfying $\mu_{X_n}=(f_n)_*\mu_{X_{n+1}}$ for every $n \in \N$. The {\it inverse limit extended metric measure space} of $\seqn{(X_n, d_{X_n}, \mu_{X_n}), f_n}$ is defined as the quadruplet $(X, \tau_X, d_X, \mu_X)$:
\begin{align*}
X= \varprojlim X_n := \defset{ (x_n) \in \prod_{n=1}^\infty X_n }{ f_n (x_{n+1}) = x_n \text{ for all } n }
\end{align*}
equipped with the subset topology $\tau_X$ induced by the product topology on $\prod_{n=1}^\infty X_n$, the {\it extended} distance $d_X (x, y)=\sup_{n} d_{X_n}(\pi_n(x), \pi_n(y)) \in [0, \infty]$, and the inverse limit measure $\mu_X$. 

\smallskip
The inverse limit is not necessarily a metric measure space in the usual sense: the limit
distance \(d_X\) may take the value \(+\infty\), the topology \(\tau_X\) is not necessarily
induced by \(d_X\), and the measure $\mu_X$ is not necessarily Borel with respect to the topology induced by $d_X$. This decoupling between the measurable/topological structure and the
metric structure leads to pathological phenomena: the support of $\mu_X$ may be empty with respect to the topology $\tau_{d_X}$ (i.e., every open metric ball is a $\mu_X$-negligible set)
because the support of \(\mu_X\) is given rather by the topology $\tau_{X}$.
Moreover, \(d_X\)-Lipschitz functions are not  necessarily 
\(\tau_X\)-Borel measurable. These pathologies naturally occur in infinite-dimensional
spaces, such as infinite-dimensional Gaussian spaces and abstract Wiener spaces.
For this reason, the usual theory of metric measure spaces cannot be applied na\"ively in the
present setting. We need a framework of extended metric measure spaces in which the topology,
measure, and extended distance are allowed to be decoupled. Several such frameworks have
been recently developed, providing useful classes of extended metric measure spaces, see, for instance,
\cite{AGS;C, AES, Sa, PTS, PT}.
%
\smallskip

\paragraph{\bf Extended metric measure spaces} According to~\cite{AGS;C},  a quadruplet $X=(X, \tau_X, d_X, \mux)$ is a \emph{Polish extended metric measure space} (for short: Polish extended mm-space) if 
\begin{enumerate}
\item  \label{d:PEMS1I} $\tau_X$ is a Polish topology on the set $X$. 
\item  \label{d:PEMS2I} $d_X :X\times X\to [0, \infty]$ is a complete extended distance that is lower semicontinuous with respect to the product topology $\tau_{X\times X}$. Here, {\it an extended distance $d_X$} means that $d_X$ satisfies the usual axioms of distance, but allowed to take the value $+\infty$.
\item  \label{d:PEMS3I} the topology $\tau_{d_X}$ induced by $d_X$ is stronger than or equal to $\tau_X$. 
\item  \label{d:PEMS4I} $\mu_X$ is a Borel probability measure on $(X, \tau_X)$.
\end{enumerate}

\smallskip
Let $X$ and $Y$ be Polish extended mm-spaces. In Definition~\ref{d:mmiso}, we introduce {\it isomorphism} between two Polish extended mm-spaces: we say that 
$X$ is {\it isomorphic} to $Y$ and write $X\cong Y$ 
if there exists a Borel map $f: X\to Y$ and a Borel set $X_0 \subset X$ with $\mux(X\setminus X_0)=0$,
\begin{align*}
d_Y (f(x), f(x')) = d_X (x, x')
\end{align*}
for any $x, x' \in X_0$, and $f_* \mux= \muy$.  In this case, we write 
$$X \cong Y.$$
Our first result is that every Polish extended mm-space arises as an inverse limit of mm-spaces.
\begin{thm}[{Theorem~\ref{prop;mmisotolim}}]\label{t:Int1}
For any Polish extended mm-space $X$,
there exists a sequence $\seqn{X_n}$ of mm-spaces with $X_n \prec X_{n+1}$ for all $n\in\N$ and
$$X\cong\varprojlim X_n.$$
\end{thm}

\paragraph{\bf Stability of Ricci curvature lower bound}
We study the stability of synthetic Ricci
curvature lower bounds and related functional inequalities under inverse limits for Polish extended mm-spaces.
Stability problems for curvature lower bounds, and for related objects such as spectra of
Laplacians, have a long history. For instance, such questions have been studied by
Fukaya~\cite{Fu}, Cheeger--Colding~\cite{CC1,CC2,CC3}, Sturm~\cite{St1,St2}, and
Gigli--Mondino--Savar\'e~\cite{GMS} in the setting of pointed measured
Gromov--Hausdorff convergence and its variants, and by
Funano--Shioya~\cite{FS} and Ozawa--Yokota~\cite{OY} in the setting of the concentration
topology.
These convergence theories, however, are not directly applicable to the extended metric measure
spaces. Due to the decoupling of the topology, the distance, and the measure,  it is not even clear in general how to formulate these
convergences in such a way that the relevant metric-measure structure is captured. 

As the second main result, under the inverse limit, we prove the
stability of the curvature-dimension condition \(\CD(K,\infty)\), the Riemannian
curvature-dimension condition \(\RCD(K,\infty)\),  the Evolution Variational Inequality~$\EVI(K,\infty)$, the log-Harnack $\LH(K)$/the dimension-free Harnack inequality~$\DFH(K)$,~the Poincar\'e inequality
\(\PI(C_P)\), the log-Sobolev inequality \(\LS(C_{LS},D_{LS})\), and the Talagrand inequality
\(\TL(C_T)\). The precise definitions of these curvature bounds and functional inequalities
are given in Section~\ref{s:FI}.
\begin{thm}[Theorem~\ref{t:PCD}, Corollary~\ref{c:RCD}, Theorems~\ref{t:P}--\ref{t:EVI}] \label{t:mthm}
Let~$\seqn{(X_n, d_{X_n}, \mu_{X_n}), f_n}$ be an inverse system of metric measure spaces and $(X, \tau_X, d_X, \mu_X)$ its inverse  limit Polish extended metric measure space.
If each $(X_n, d_{X_n}, \mu_{X_n})$ satisfies one of the conditions
\begin{quote}
$\CD(K, \infty)$, $\RCD(K, \infty)$,  $\EVI(K, \infty)$,
$\LH(K)$, $\DFH(K)$, $\PI(C_P)$,\\ $\LS(C_{LS}, D_{LS})$,  and $\TL(C_T)$,
\end{quote}
then so does the limit~$(X, \tau_X, d_X, \mux)$.
\end{thm}
We note that, unlike in the classical metric measure setting, it is not presently known whether
\(\RCD\) implies the \(\EVI\) property for the heat flow in general
extended metric measure spaces. Consequently, in the present framework the \(\EVI\) property
cannot be obtained merely by combining the \(\CD\) condition with infinitesimal
Hilbertianity. We therefore discuss the stability of \(\EVI\) as a separate statement. 

\paragraph{\bf Application} As applications, we establish \(\RCD\) and \(\EVI\) for several
infinite-dimensional spaces, including the infinite product Gaussian space, classical Wiener
spaces, and their quotient spaces; see Theorems~\ref{t:CQ} and~\ref{t:WPE}. These applications are closely related to earlier work on curvature bounds for Wiener spaces. Fang--Shao--Sturm~\cite[Theorem~1.5]{FSS} proved
\(\CD(1,\infty)\) for the Wiener space. Moreover, Ambrosio--Erbar--Savar\'e~\cite{AES}
obtained an \(\EVI\)-type property for abstract Wiener spaces with respect to the dynamic transport-type
extended distance $W_{\mathcal E}$ associated with the Ornstein-Uhlenbeck Dirichlet form $\mathcal E$. It is currently not known whether $W_{\mathcal E}$ is identical to the static \(2\)-Wasserstein
extended distance \(W_{2,d}\) induced by the Cameron--Martin distance \(d\) used in the present paper.
To the best of our
knowledge,  the \(\EVI(1,\infty)\) property with respect to this static distance \(W_{2,d}\) on Wiener space has not been explicitly discussed in the literature although the aforementioned work provides strong evidence for it.
Our stability result gives this statement directly from the $\EVI$ property of the finite-dimensional Gaussian spaces, and also
extends it to quotient spaces of the Wiener space.

\smallskip
\paragraph{\bf The role of the topology}
The role of the topology in an extended metric measure space is not as essential as the one in ordinary metric measure spaces. In the ordinary setting, the metric topology induces the
Borel $\sigma$-algebra and the support of the measure.  In the extended setting,  the topology \(\tau_X\) provides the measurable structure and  the support of
\(\mu_X\),
while the extended distance \(d_X\) encodes the metric geometry. This separation explains why
our isomorphism does not necessarily preserve the topology. The role of the topology is rather supportive for the measurable structure,
but the metric-measure properties studied in this paper essentially depend only on the measure and the
extended distance up to null sets. Hence a measure-preserving map that preserves the extended
distance on a set of full measure is sufficient to transport these properties. In
Theorem~\ref{thm:energy} and Proposition~\ref{p:SI}, we prove  that our
isomorphisms preserve the Cheeger energy, the heat semigroup, the Wasserstein distance, the curvature-dimension conditions as well as the relevant functional inequalities. 
\smallskip
\paragraph{\bf Organisation} 
The paper is organised as follows. In Section~\ref{s:invlim}, we recall the relevant notions of
extended metric measure spaces and show that inverse limits of metric measure spaces fit into two
frameworks: they are Polish extended metric measure spaces in the sense of~\cite{AGS;C}
(Proposition~\ref{p:IPE}) and extended metric measure spaces in the sense of~\cite{AES}
(Proposition~\ref{prop;invlimextsp}). We then introduce a notion of isomorphism between Polish
extended metric measure spaces (Definition~\ref{d:mmiso}) and prove that every Polish extended
metric measure space is isomorphic to an inverse limit of metric measure spaces
(Theorem~\ref{prop;mmisotolim}). In Section~\ref{sec:CD}, we prove the stability of the
\(\CD\) condition under inverse limits. In Section~\ref{sec:Mosco}, we establish Mosco
convergence of the Cheeger energies associated with the finite-dimensional approximating
distances. In Section~\ref{s:FI}, we prove the stability of the \(\RCD\) condition and
several functional inequalities, including the log-Harnack inequality~\(\LH\), the dimension-free
Harnack inequality~\(\DFH\), the Poincar\'e inequality~\(\PI\), the log-Sobolev inequality~\(\LS\),
and the Talagrand inequality~\(\TL\). The precise definitions of these curvature and functional
inequality conditions are given there. In Section~\ref{sec:Inv}, we prove that the Cheeger energy is
invariant under isomorphisms of Polish extended metric measure spaces. In Section~\ref{sec:Sta},
we show that the curvature conditions and functional inequalities are also
invariant under such isomorphisms. Finally, in Section~\ref{ssec:WS}, we discuss examples,
including infinite product spaces, abstract Wiener spaces, and their quotient spaces.

\section{Inverse limits and extended metric measure spaces} \label{s:invlim}
In this section, we recall two definitions of extended metric measure spaces and show that inverse limit of metric measure spaces satisfy both of them.

\subsection{Extended metric measure spaces and inverse limits}
We say that $d_X: X \times X \to [0,\infty]$ is {\it an extended distance} (or {\it an extended metric})  if it satisfies the axioms of distance functions except that it may take $+\infty$. 
We say that $(X, d_X)$ is {\it complete} if every $d_X$-Cauchy sequence converges in $(X, d_X)$.  
\begin{defn}[Polish extended metric measure space {\cite[Definition~2.3]{AGS;C}}] \label{d:PEMS}
We call a quadruplet $X=(X, \tau_X, d_X, \mux)$ a \emph{Polish extended metric measure space} (for short: Polish extended mm-space) if 
\begin{enumerate}
\item  \label{d:PEMS1} $\tau_X$ is a Polish topology on the set $X$. 
\item  \label{d:PEMS2} $d_X :X\times X\to [0, \infty]$ is a complete extended distance that is lower semicontinuous with respect to the product topology $\tau_{X\times X}$.
\item  \label{d:PEMS3} the topology $\tau_{d_X}$ induced by $d_X$ is stronger than $\tau_X$. 
\item  \label{d:PEMS4} $\mu_X$ is a Borel probability measure on $(X, \tau_X)$.
\end{enumerate}
When $d_X:X\times X\to [0, \infty)$, $\tau_X=\tau_{d_X}$ and $\supp_{\tau_X}\mu_X=X$, we call $(X, \tau_X, d_{X}, \mu_X)$ a {\it metric measure space} (for short: mm-space).
In this  case, we shorten the notation $X=(X, \tau_X, d_X, \mux)$ to $X=(X, d_X, \mux)$.
\end{defn}
We note that the original definition in Ambrosio--Gigli--Savar\'{e}~\cite[Definition~2.3]{AGS;C} allows $\mu_X$ to have infinite total mass $\mu_X(X)=+\infty$. 
The following proposition gives a sufficient condition for $(X, \tau_{d_X}, d_X, \mu_X)$ (with the topology $\tau_{d_X}$, not $\tau_X$) to be  an mm-space .
\begin{prop} \label{p:PM}
Let $(X, \tau_X, d_X, \mux)$ be a Polish extended mm-space.
Suppose that $d_X (x, x')<\infty$ for any $x, x' \in X$, $\supp_{\tau_{d_X}}\mu_X=X$ and that  $(X, \tau_{d_X})$ is separable.
Then $(X, \tau_{d_X}, d_X, \mux)$ is an mm-space.
\end{prop}
\begin{proof}
It suffices to check \eqref{d:PEMS4}  in Definition~\ref{d:PEMS} with $\tau_{d_X}$ in place of $\tau_X$. We show that  the Borel $\sigma$-algebra $\mathcal B(\tau_{d_X})$ generated by $\tau_{d_X}$ coincides with $\mathcal{B}(\tau_X)$. Since the topology $\tau_{d_X}$ is finer than $\tau_{X}$ due to the property~\eqref{d:PEMS3} in Definition~\ref{d:PEMS}, we have $\mathcal{B}(\tau_X)\subset\mathcal{B}(\tau_{d_X})$.  Since $\tau_{d_X}$ is countably generated by open metric balls due to the separability hypothesis, it suffices to show  that open metric balls are Borel with respect to $\tau_X$.  Closed metric balls are sublevel sets of the lower semicontinuous function $d_X(x_0, \cdot)$ for a fixed point $x_0 \in X$,  hence, they are $\tau_X$-closed. By $B_X(x, r)=\cup_{m \in \N} \{d_X(x, \cdot) \le r-1/m\}$, open balls are $\tau_X$-Borel. 
\end{proof}

For two metric measure spaces $X$ and $Y$, we write $Y \prec X$, and say that {\it $Y$ is below $X$ in the $1$-Lipschitz order} if there exists a Borel measurable map $f: X \to Y$ such that $f$ is  $1$-Lipschitz and measure-preserving. Namely, 
$$d_Y(f(x_1), f(x_2)) \le d_X(x_1, x_2) \quad  \forall x_1, x_2 \in X \quad \text{and} \quad f_*\mu_X =\mu_Y \comma$$
where $f_*\mu_X$ is the push-forward measure of $\mu_X$ by $f$.
\begin{defn}[Inverse limit extended mm-space]\label{def;invlimmeas}
Let $\seqn{X_n}$ be a sequence of mm-spaces with $X_n \prec X_{n+1}$ for all $n$
and $f_n :X_{n+1}\to X_n$ a $1$-Lipschitz map with $(f_n)_* \mu_{X_{n+1}} =\mu_{X_n}$.
Define the set
\[
X= \varprojlim X_n := \defset{ (x_n) \in \prod_{n=1}^\infty X_n }{ f_n (x_{n+1}) = x_n \text{ for all } n }
\]
and equip $X$ with the relative topology $\tau_X$ of the product topology on $\prod_{n=1}^\infty X_n$ and with the projective limit measure $\mux$,
i.e., the unique Borel probability measure on $(X, \tau_X)$ satisfying $(\pi_n)_* \mux = \mu_{X_n}$ for all $n$, where $\pi_n :X\to X_n$ is the projection. (This is a standard application of the Kolmogorov extension theorem for projective limit.) 
Furthermore, define
\begin{align}\label{eq;dx}
d_X(x, x')
:= \limn d_{X_n} (x_n, x'_n)
= \limn d_{X_n} (\pi_n(x), \pi_n(x'))
\in [0, \infty]
\end{align}
for $x=(x_n), x'=(x'_n) \in X$.
We call $(X, \tau_X, d_X, \mux)$ the \emph{inverse limit} of $\seqn{(X_n, f_n)}$.
\end{defn}

Note that $d_{X_n} (x_n, x'_n)$ is monotone non-decreasing due to the $1$-Lipschitz property of $f_n$, and therefore, $\limn d_{X_n} (x_n, x'_n)$ is well-defined (possibly $+\infty$).  

\begin{prop}\label{p:IPE}
Let $\seqn{X_n}$ be a sequence of mm-spaces with $X_n \prec X_{n+1}$ for all $n$.
Then the inverse limit $X:=\varprojlim X_n$
is a Polish extended mm-space in the sense of Definition~\ref{d:PEMS}.
\end{prop}
\begin{proof}
As Property~\eqref{d:PEMS4} follows from Definition~\ref{def;invlimmeas},  and Property~\eqref{d:PEMS3} follows by the definitions of $d_X$ and the product topology, we only verify Properties~\eqref{d:PEMS1} and~\eqref{d:PEMS2} in Definition~\ref{d:PEMS}. 

We prove that $(X, d_X)$ is complete. Let $\{y^k\}_{k \in \N} \subset X$ be a $d_X$-Cauchy sequence. Let $y_n^k:=\pi_n(y^k) \in X_n$. By $d_{X_n} (y_n^k, y_n^l) \le d_{X} (y^k, y^l)$, the sequence $\{y_n^k\}_{k \in \N}$ is $d_{X_n}$-Cauchy for every $n \in \N$. Since $(X_n, d_{X_n})$ is complete by hypothesis, the limit point $y_n \in X_n$ exists for every $n \in \N$. Set $y:=(y_n)_{n \in \N}$. Since the map $f_n: X_{n+1} \to X_n$ witnessing $X_n \prec X_{n+1}$ is $1$-Lipschitz, we obtain that 
\begin{align} \label{e:IPE-1}
f_n(y_{n+1})= \lim_{k \to \infty} f_n(y_{n+1}^k) =\lim_{k\to \infty}y_n^k= y_{n} \fstop
\end{align}
Thus, $y \in X$. It remains to show that \(y^k\to y\) with respect to \(d_X\). Let \(\varepsilon>0\).
Since \(\{y^k\}_k\) is \(d_X\)-Cauchy, there exists \(N\in\mathbb N\) such that
\[
d_X(y^k,y^\ell)<\varepsilon
\qquad\text{for all }k,\ell\ge N.
\]
Fix \(k\ge N\). Then, for every \(n\in\mathbb N\),
\[
d_{X_n}(y^k_n,y_n)
=
\lim_{\ell\to\infty}d_{X_n}(y^k_n,y^\ell_n)
\le
\varepsilon.
\]
Taking the supremum over \(n\), we obtain
\[
d_X(y^k,y)\le\varepsilon.
\]
Thus \(y^k\to y\) in \(d_X\), and \((X,d_X)\) is complete.

We now prove that $(X, \tau_X)$ is Polish. Since the product topology is Polish whenever $(X_n, \tau_{X_n})$ is Polish for every $n \in \N$, the product space $\prod_{n=1}^\infty X_n$ equipped with the product topology is Polish. By recalling that every closed set in a Polish space is again Polish with respect to the relative topology, it suffices to verify that $X$ is closed as a subset in the product space $\prod_{n=1}^\infty X_n$. Since a subset in a Polish space is closed if and only if it is sequentially closed, we only need to show that $X$ is sequentially closed.  
 Let $\{y^k\}_{k \in \N} \subset X$ be a $\tau_X$-converging sequence. We keep the same notation as in the first paragraph. Then, by the definition of the product topology and the same proof as in the previous paragraph regarding the completeness, we can verify \eqref{e:IPE-1} and, therefore, $X$ is closed. Thus, $(X, \tau_X)$ is Polish. 
The lower semicontinuity of $d_X$ is a direct consequence of the expression 
\begin{align*}
d_{X}(x, y)=\sup_{n \in \N}d_{X_n}(\pi_n(x), \pi_n(y))
\end{align*}
and the general fact that the monotone increasing limit of a sequence of continuous functions is lower semicontinuous. 
\end{proof}

Ambrosio--Erbar--Savar\'{e}~\cite{AES} introduced another definition  of extended mm-spaces. We say that $d_X: X \times X \to [0,\infty)$ is {\it semidistance} if it satisfies the axioms of distance functions except that $d_X(x,y)=0 \implies x=y$ is not required.
\begin{defn}[Extended mm-space~{\cite{AES}}]\label{Def:ext-mtmp}
We call $X=(X, \tau_X, d_X, \mux)$ an \emph{extended metric measure space} (for short: extended mm-space) if 
\begin{enumerate}
\item\label{Def:ext-mtmp1} $\tau_X$ is a Hausdorff topology on $X$.
\item\label{Def:ext-mtmp2} There exist $\tau_X^{\times 2}$-continuous bounded semidistances $d_i$ on $X$ with $\iiI$ and 
$$d_X (x, x') = \sup_{\iiI} d_i (x, x') \qquad x, x' \in X. $$
\item\label{Def:ext-mtmp3} $\tau_X$ is generated by 
$$\lip_b (X, \tau_X, d_X) := \defset{ f: X\to\R }{ f \text{ is bounded $\tau_X$-continuous $d_{X}$-Lipschitz} }\fstop$$
\item\label{Def:ext-mtmp4} $\mux$ is a Radon probability measure on $(X, \cB(\tau_X))$,
where $\cB(\tau_X)$ is the Borel $\sigma$-algebra of $(X, \tau_X)$.
\end{enumerate}
\end{defn}

\begin{prop}\label{prop;invlimextsp}
Let $\seqn{X_n}$ be a sequence of mm-spaces with $X_n \prec X_{n+1}$ for all $n$.
Then the inverse limit $X:=\varprojlim X_n$ is an extended mm-space  in the sense of Definition~\ref{Def:ext-mtmp}
\end{prop}
\begin{proof}
Property (1) follows from Proposition~\ref{p:IPE}, and property (4) follows because every Borel
probability measure on a Polish space is Radon. We verify property~ \eqref{Def:ext-mtmp2}. For \(M\in\mathbb N\),
set
\[
d_{n,M}(x,y):=d_{X_n}(\pi_n(x),\pi_n(y))\wedge M.
\]
Then \(d_{n,M}\) is a bounded continuous semidistance on \(X\), and
\[
d_X(x,y)=\sup_{n,M\in\mathbb N} d_{n,M}(x,y).
\]
Thus property (2) holds. 
It remains to prove \eqref{Def:ext-mtmp3}. We first prove that  the algebra $\lip_b (X, \tau_X, d_X)$ separates points in $X$. Let $x \neq y \in X$. 
Then there exists $n_0 \in \N$ so that $d_{X_n}(\pi_n(x), \pi_n(y))>0$ for any $n \ge n_0$. Thus, for a positive constant $M$, the function $\rho_{n}:=d_{X_n}(\pi_n(x), \pi_n(\cdot)) \wedge M \in \lip_b (X, \tau_X, d_X)$ separates $x$ and $y$ for any $n \ge n_0$ since 
$$\rho_n (x) =0 < \rho_n (y)\fstop$$
We now prove that the Lipschitz algebra generates the topology. Note that the product topology $\tau_X$ has a base of the form $\pi_n^{-1}(B_{X_n}(a, r))$ for $a \in X_n$, $r>0$ and $n \in \N$, where $B_{X_n}(a, r):=\{x \in X_n: d_{X_n}(a, x)<r\}$. We take $f_{n,a}(x):=d_{X_n}(\pi_n(x), a) \wedge M \in \lip_b(X, \tau_X, d_X)$. Noting $\pi^{-1}(B_{X_n}(a, r)) = f_{n, a}^{-1}((-\infty, r))$ with $M>r$, we conclude that every open set is generated by $\lip_b(X, \tau_X, d_X) $, hence $\tau_X$ is generated by $\lip_b(X, \tau_X, d_X) $. 
\end{proof}

\subsection{Isomorphism of extended metric measure spaces}
We introduce a concept of {\it isomorphism} for extended metric measure spaces. 
\begin{defn}\label{d:mmiso}
Let $X$ and $Y$ be Polish extended mm-spaces.
We say that $X$ is {\it isomorphic} to $Y$ and write $X\cong Y$
if there exists a Borel map $f: X\to Y$ and a Borel set $X_0 \subset X$ with $\mux(X\setminus X_0)=0$,
\begin{align}\label{eq;isom}
d_Y (f(x), f(x')) = d_X (x, x')
\end{align}
for any $x, x' \in X_0$, and $f_* \mux= \muy$. 
\end{defn}

\begin{rem}
The isomorphism defined in Definition~\ref{d:mmiso} does not see the topology. Nonetheless, this isomorphism preserves the synthetic Ricci curvature lower bound as well as related functional inequalities in extended metric measure spces, which we will discuss in Proposition~\ref{p:SI}.
In the framework of extended metric measure spaces, requiring further topological condition (e.g.~homeomorphism) is too strong for our purpose. In fact, in the case of abstract Wiener spaces~$(W, H, \mu)$ (see Section~\ref{ssec:WS}), any two abstract Wiener spaces with isometrically isomorphic Cameron–Martin Hilbert spaces are isomorphic in the sense of Definition~\ref{d:mmiso}, while there are many abstract Wiener spaces that are not necessarily homeomorphic to each other as the choice of the ambient Banach space $W$ is not topologically unique, e.g., $W$ can be taken as the space $C([0,1], \R)$ of continuous paths on $\R$ as well as $L^2([0,1], dx)$. See, e.g.,~\cite[Sec.~3.9 in Chap.~3]{Bo} for further details.
\end{rem}

We see that ``$\cong$" is an equivalence relation on the collection of all Polish extended metric measure spaces.
\begin{prop} \label{p:S}
The relation $X \cong Y$ is an equivalence relation. 
\end{prop}
\begin{proof}
Reflexivity is immediate, and transitivity follows by composition.
It remains only to prove symmetry~$X\cong Y \iff Y\cong X$.
Let $f$ and $X_0$ be a Borel map and a Borel set as in Definition~\ref{d:mmiso} respectively.
Fix a point $x_0 \in X$, and set~$Y_0 :=f(X_0)\subset Y$. Define a map $g: Y \to X$ as 
\begin{align*}
g(y):=
\begin{cases}
\ f^{-1}(y) &\quad \text{if } y \in Y_0 ;\\
\ x_0 & \quad \text{otherwise} \fstop
\end{cases}
\end{align*}
Since $f|_{X_0} :X_0 \to Y$ is injective, $Y_0\subset Y$ is a Borel set (see e.g.~Srivastava~\cite[Theorem~4.5.4, Proposition~4.5.1]{Sr}) and $g$ is a Borel map.

For \(y,y'\in Y_0\), choose \(x,x'\in X_0\) such that \(y=f(x)\) and \(y'=f(x')\). Then
\[
d_X(g(y),g(y'))
=
d_X(x,x')
=
d_Y(f(x),f(x'))
=
d_Y(y,y').
\]
Thus \(g\) preserves the distance on the full-measure set \(Y_0\).

Since \(\mu_Y(Y_0)=1\) and \(g\circ f=\mathrm{id}\) on \(X_0\), for every Borel set
\(A\subset X\),
\[
g_*\mu_Y(A)
=
\mu_Y(g^{-1}(A))
=
\mu_Y(f(A\cap X_0))
=
\mu_X(A\cap X_0)
=
\mu_X(A). \qedhere
\]
\end{proof}

In terms of the isomorphism defined in Definition~\ref{d:mmiso}, we can show that {\it any} Polish extended metric measure space can be constructed as the inverse limit of metric measure spaces. The following lemma is essential for the proof. 
\begin{lem}\label{lem;weaksep}
Let $X$ be a Polish extended metric measure space.
Then there exist a Borel set $K\subset X$ with $\mux(K)=1$ and $1$-Lipschitz Borel functions $\phi_n :X\to[0, \infty)$ with
\begin{align}\label{eq;weaksep}
d_X (x, x') = \sup_\niN |\phi_n (x)-\phi_n (x')|
\end{align}
for any $x, x'\in  K$.
\end{lem}
We first prove a preparatory lemma.
\begin{lem}\label{l:LSD}
Let $A \subset X$ be a $\tau_X$-compact subset and define $\rho_A(\cdot):=d_X(\cdot, A):=\inf_{x \in A}d_X(\cdot, x)$. Then, $\rho_A$ is $\tau_X$-lower semicontinuous.
\end{lem}
\begin{proof}
Take any sequence $x_m \xrightarrow{\tau_X}x$ and $L:=\liminf_{m \to \infty}\rho_A(x_m)$. We can take a (non-relabelled) subsequence of $(x_m)$ such that $L=\lim_{m \to \infty}\rho_A(x_m)$. For $m \in \N$, choose $a_m$ such that $d_X(x_m, a_m) \le \rho_A(x_m) + \frac{1}{m}$. Since $A$ is compact, we can extract a subsequence $a_{m_k} \xrightarrow{\tau_X} a \in A$. By the joint lower semicontinuity of $d_X$, we have 
$$d_X(x, a) \le \liminf_{k \to \infty}d_X(x_{m_k}, a_{m_k}).$$
By definition, $\rho_A(x) \le d_X(x, a)$. Thus, we get 
$$\rho_A(x) \le d_X(x, a ) \le  \liminf_{k \to \infty}d_X(x_{m_k}, a_{m_k}) \le  \liminf_{k \to \infty}(\rho_A(x_{m_k}) + \frac{1}{m_k}) =  \liminf_{k \to \infty}\rho_A(x_{m_k}) =L. \qedhere$$ 
\end{proof}
\begin{proof}[Proof of Lemma~\ref{lem;weaksep}]
For $i\in\N$, we take a compact set $K_i \subset X$ with $\mux(K_i)>1-(1/i)$.
We know that the Borel set $K:=\bigcup_{i\in\N} K_i \subset X$ satisfies $\mux(K)=1$.
We also take a countable dense set $D=\{x_j : j\in\N\}\subset K$ in $(K, \tau_X)$
and a complete metric $d_\tau :X\times X\to[0, \infty)$ which metrises $\tau_X$,
and let $L:=\Q\cap(0, \infty)$.
We use the notation
\[
\uB(x, r) := \{x' \in K : d_\tau (x, x')\le r\}
\]
for $x\in X$ and $r>0$.

For $i, j, k\in\N$ and $l\in L$, we define the function
\[
\phi_{i, j, k, l} :X\to[0, \infty), \quad \phi_{i, j, k, l} (x)= \min\{ d_X (x, \uB(x_j, 1/k)\cap K_i), l \},
\]
where we set $\phi_{i, j, k, l}(x)=l$ if $\uB(x_j, 1/k)\cap K_i =\emptyset$,
and rearrange them as
\[
\{\phi_{i, j, k, l} : i, j, k, l\in\N\} = \{\phi_n : n\in\N\}.
\]
Due to Lemma~\ref{l:LSD}, $\phi_n$ is Borel. 
We fix $x, x' \in K$ and $l\in L$ with $d_X (x, x')>l$.
Then, due to the $\tau_X$-lower semicontinuity of $d_X$,  there exists $\delta>0$ with $d_X (x'', x')>l$ for any $x''\in\uB(x, \delta)$.
We take $k\in\N$ with $2/k <\delta$, $x_j \in\uB(x, 1/k)$, and $i\in\N$ with $x\in K_i$.
Then 
\begin{align*}
\phi_{i, j, k, l} (x) &\le d_X (x, \uB(x_j, 1/k)\cap K_i) =0, \\
d_X (x', \uB(x_j, 1/k)\cap K_i) &\ge d_X (x', \uB(x_j, 1/k)) \ge d_X (x', \uB(x, \delta)) \ge l,
\end{align*} 
and hence
\[
|\phi_{i, j, k, l} (x)-\phi_{i, j, k, l} (x')| = \phi_{i, j, k, l} (x') =l.
\]
Since \(l<d_X(x,x')\) was arbitrary, we obtain
\[
\sup_n |\phi_n(x)-\phi_n(x')|\ge d_X(x,x').
\]

The reverse inequality follows from the fact that each \(\phi_n\) is \(1\)-Lipschitz with respect
to \(d_X\). Therefore \eqref{eq;weaksep} holds.
\end{proof}

In the following theorem, we show that any Polish extended metric measure space can be obtained as inverse limits of metric measure spaces. 
\begin{thm}\label{prop;mmisotolim}
For any Polish extended metric measure space $X$,
there exists a sequence $\seqn{X_n}$ of mm-spaces with $X_n \prec X_{n+1}$ for all $n\in\N$ and
$$X\cong\varprojlim X_n$$ in the sense of Definition~\ref{d:mmiso}.
\end{thm}

\begin{proof}
Let $\phi_n$ be Borel functions in Lemma~\ref{lem;weaksep}.
Let $\Phi_n: X \to \R^{n}$ be a Borel map defined by $\Phi_n(x):=(\phi_1(x), \phi_2(x),\ldots, \phi_n(x))$. 
Define 
$$\|(x_i)_{1 \le i \le n}-(y_i)_{1 \le i \le n}\|_{\infty, n}:=\max_{1 \le i \le n}|x_i-y_i|$$
for $(x_i)_{1 \le i \le n}, (y_i)_{1 \le i \le n} \in\R^n$ and define $X_n \subset \R^n$ by 
$$X_n:=\Bigl(\supp\bigl[ (\Phi_n)_*\mu_X\bigr], \|\cdot\|_{\infty, n}, (\Phi_n)_*\mu_X\Bigr).$$

We have $X_n \prec X_{n+1}$ with the projection map $\pi_{n+1, n}: \R^{n+1}\to \R^n$ restricted on $X_{n+1}$, thus $\seqn{(X_n, \pi_{n+1, n})}$ is an inverse system of mm-spaces.
Define $K' =\cap_{n \ge 1}\Phi_n^{-1}(X_n)$, which satisfies $\mu_X(K')=1$. 
Let $X'=(X', \tau_{X'}, d_{X'}, \mu_{X'}):=\varprojlim X_n$. Define $\Phi: X \to X' \subset\prod_{n=1}^\infty \R^n$ by 
\begin{align*}
\Phi(x)=
\begin{cases}
(\Phi_i(x))_{i=1}^\infty \quad & \text{if $x \in K'$}
\\
x_0 \quad & \text{otherwise},
\end{cases}
\end{align*}
where $x_0 \in X'$ is a fixed point. 
Then, the map $\Phi$ is Borel. For every \(n\),
\[
(\pi_n')_*\Phi_*\mu_X=(\Phi_n)_*\mu_X=\mu_{X_n},
\]
where \(\pi_n':X'\to X_n\) is the canonical projection. Since the inverse-limit measure is
uniquely determined by its finite-dimensional marginals, \(\Phi_*\mu_X=\mu_{X'}\).

Let \(K\subset X\) be the full-measure Borel set from Lemma~\ref{lem;weaksep}, and set
\[
\Omega:=K\cap K'.
\]
Then \(\mu_X(\Omega)=1\). For \(x,y\in\Omega\), we have
\[
\begin{aligned}
d_{X'}(\Phi(x),\Phi(y))
&=
\sup_n d_{X_n}(\Phi_n(x),\Phi_n(y))\\
&=
\sup_n \max_{1\le i\le n}|\phi_i(x)-\phi_i(y)|\\
&=
\sup_i |\phi_i(x)-\phi_i(y)|\\
&=
d_X(x,y).
\end{aligned}
\]
Thus \(\Phi\) is measure preserving and distance preserving on the full-measure Borel set
\(\Omega\). Hence \(X\cong X'\) in the sense of Definition~\ref{d:mmiso}.
%
%
\end{proof}

\section{Curvature-dimension condition for inverse limits} \label{sec:CD}
In this section, we prove the curvature-dimension condition ${\rm CD}(K,\infty)$ for inverse limits.  
Let $X=(X, \tau_X, d_X, \mu_X)$ be a Polish extended mm-space. Let $\mathcal P(X)$ denote the space of all Borel probability measures on $X$ and $\Pac(X):=\{\nu \in \mathcal P(X): \nu \ll \mu_{X}\}$ denote the subspace consisting of all absolutely continuous measures with respect to the reference measure $\mu_X$. We define the {\it relative entropy} $\ENT_{\mu_X}: \Pac(X) \to [0, \infty]$  as 
$$\Ent{\mu_X}{\nu}:=\int_{X} \rho \log \rho \diff \mu_X \quad \text{for $\nu \in \Pac$ with $\rho=\frac{\diff \nu}{\diff \mu_X}$} \fstop$$
We define the domain of $\ENT_{\mu_X}$ by $D(\ENT_{\mu_X}):=\{\nu \in \Pac(X): \Ent{\mu_X}{\nu}<\infty\}$. 
For $\nu, \sigma \in \mathcal P(X)$, define the space ${\rm Cpl}(\nu, \sigma) \subset \mathcal P(X \times X)$ of couplings by  
$${\rm Cpl}(\nu, \sigma) :=\{\c \in  \mathcal P(X \times X): \c(A \times X) = \nu(A) \comma \quad \c(X \times A) = \sigma(A)\comma  \quad \forall A \in \mathcal B(X)\} \fstop$$
For $\nu, \sigma \in \mathcal P(X)$, define {\it $2$-Wasserstein extended distance} as  
\begin{align} \label{d:Was}
W_{2, d_X}(\nu, \sigma)^2:=\inf_{c} \int_{X \times X} d_X(x, y)^2 \diff\c(x, y) \in [0, \infty]\comma
\end{align}
where the infimum is taken over all $\c \in {\rm Cpl}(\nu, \sigma)$. By e.g. \cite[Theorem~4.1]{Vi08} for lower semicontinuous cost functions (possibly taking $+\infty$), there exists an optimal coupling that attains the infimum of the right-hand side of \eqref{d:Was}. We denote by ${\rm OptCpl}_{d_{X}^2}( \nu, \sigma)$ the set of optimal couplings in ${\rm Cpl}(\nu, \sigma)$ with respect to the cost $d_{X}^2$. 

\begin{defn}[{\cite[Definition~9.1]{AGS;C}}]\label{def;CD}
Let $K \in \R$ and $X=(X, \tau_X, d_X, \mu_X)$ be a Polish extended mm-space.
We say that $X$ satisfies the \emph{$\CD(K, \infty)$ condition}
if for any $\nu_0, \nu_1\in\DomEnt{X}$ with $W_{2, d_X} (\nu_0, \nu_1)<\infty$
there exists a time-parametrised family $\{\nu_t\}_{t\in[0, 1]} \subset D(\ENT_{\mu_X})$  such that 
\begin{align*}
&W_{2, d_X} (\nu_s, \nu_t) = |s-t| W_{2, d_X}(\nu_0, \nu_1), \\
&\Entm{X}{\nu_t} \le (1-t) \Entm{X}{\nu_0} + t \Entm{X}{\nu_1} -\frac{K}{2} t(1-t) W_{2, d_X}(\nu_0, \nu_1)^2
\end{align*}
for any $s, t\in [0, 1]$.
\end{defn}

The main theorem of this section is the following: 
\begin{thm} \label{t:PCD}
Let $K \in \R$ and $\seqn{X_n}$ be a sequence of mm-spaces with $X_n \prec X_{n+1}$ for all $n$
and $X:=\varprojlim X_n$ its inverse limit extended mm-space. 
If $X_n$ satisfies CD$(K, \infty)$, then so does $X$.
\end{thm}

\subsection{Proof of Theorem \ref{t:PCD}}
We separate the proof in several steps. 
As the following statement will be frequently used below, we recall the precise statement here.
\begin{thm}[{\cite[Theorem 5.1]{AES}}] \label{t:AES}
Let $X=(X, \tau_X, d_X, \mu_X)$ be an extended mm-space in the sense of Definition~\ref{Def:ext-mtmp}.
Let $I$ be a directed set and $\{\nu_i\}_{i \in I}, \{\sigma_i\}_{i \in I} \subset \mathcal P(X)$
weakly converge to $\nu, \sigma$ respectively. Then, for any choice $\{\c_i\}_{i \in I} \subset {\rm Cpl}(\nu_i, \sigma_i)$, one has:
\begin{enumerate}
\item the net $\{\c_i\}_{i \in I}$ has a limit point, say, $\c$, with respect to the weak convergence and $\c$ belongs to ${\rm Cpl}(\nu, \sigma)$;
\item if $\c_i$ converges weakly to $\c$, and $a_i: X \times X \to [0, \infty]$ is a monotone nondecreasing family of $(\tau_X \times \tau_X)$-lower semicontinuous function, then 
$$\liminf_{i \to \infty}	\int_X a_i \diff\c_i \ge \int_X a \diff\c, \quad \text{ where $a:=\sup_{i \in I} a_i$}\fstop$$
\end{enumerate}
\end{thm}

In the following lemmas, we always assume the same condition as in Theorem~\ref{t:PCD}. 
Recall that $\pi_{n}: X\to X_n$ denotes the canonical projection. 
\begin{lem}\label{l;weakconv}
For~$\nu\in D(\ENT_{\mu_X})$, define~$\nu_n :=(\rho_n \circ\pi_n)\mu_X \in\P(X)$,
where $\rho_n \in L^1 (\mu_{X_n})$ is the function such that $\underline{\nu}_n:=(\pi_n)_* \nu=\rho_n \mu_{X_n}\in\P(X_n)$.
Then $\seqn{\nu_n}$ converges weakly to $\nu$.
\end{lem}

We give two proofs of Lemma~\ref{l;weakconv} for the convenience of the reader.
\begin{proof}[Proof of Lemma~\ref{l;weakconv} by Martingale Convergence]
Let $d\nu=\rho d\mu_X$.
Then, $\rho_{n}\circ \pi_n$ is the conditional expectation of $\rho$ with respect to the filtration of the $\sigma$-algebra $\cB_n(X)=\pi_n^{-1} \cB(X_n)$. Indeed, for any $A_n \in \cB(X_n)$, we have that 
\begin{align*}
 \int_{\pi_n^{-1}(A_n)} \rho_{n}\circ \pi_n \diff \mu_{X} = \int_{A_n} \rho_{n} \diff \mu_{X_n} = \underline{\nu}_{n}(A_n)= \nu(\pi_n^{-1}(A_n)) =  \int_{\pi_n^{-1}(A_n)} \rho \diff \mu_X \fstop
\end{align*}
Since $\sigma(\cup_{n \in \N} \cB_n(X))=\cB(X)$, the Martingale Convergence yields that $\rho_{n}\circ \pi_n$ converges to $\rho$ in $L^1(X, \mu_X)$. Thus, $\nu_n$ converges weakly to $\nu$ in $\P(X)$.
\end{proof}

\begin{proof}[Proof of Lemma~\ref{l;weakconv} by Portmanteau theorem]
We start with the following.
\begin{clm}\label{clm;weakconv}
Suppose that $A\subset X$ is closed, $K\subset X$ is compact, and $A\cap K=\emptyset$.
Then
\[
\limin d_{X_n} (\pi_n (A), \pi_n (K))>0.
\]
\end{clm}
\begin{proof}
Suppose that there exist sequences $\seqn{x_n}$ and $\seqn{y_n}$ of points with
$x_n \in A$ and $y_n \in K$ for all $n$ and $\limin d_{X_n} (\pi_n (x_n), \pi_n (y_n))=0$.
We may assume that a subsequence $\seqk{y_{n(k)}}$ of $\seqn{y_n}$ converges to some $y\in K$ and that
$d_{X_{n(k)}} (\pi_{n(k)}(x_{n(k)}), \pi_{n(k)}(y_{n(k)})) \to 0$ as $\kti$. 
Then $d_{X_n} (\pi_n (y_{n(k)}), \pi_n (y))\to 0$. Thus,  $d_{X_n} (\pi_n (x_{n(k)}), \pi_n (y))\to 0$ as $\kti$ for any $n$.
This means that $\seqk{x_{n(k)}}$ converges to $y$, hence $y\in A$,
but this is a contradiction.
\end{proof}

Let $A\subset X$ be a closed set and $\ep>0$.
Take a compact set $K\subset X\setminus A$ with $\nu(A\sqcup K)>1-\ep$. 
Then $\overline{\pi_n (A)}\cap\pi_n (K)=\emptyset$ by Claim~\ref{clm;weakconv} and
\begin{align*}
\nu_n (A)
&\le \nu_n (\pi_n^{-1}(\overline{\pi_n (A)})) = \int_{\pi_n^{-1}(\overline{\pi_n (A)})} {\rho}_n \circ\pi_n \,d\mux
= \underline{\nu}_n (\overline{\pi_n (A)})
= \nu(\pi_n^{-1} (\overline{\pi_n (A)})) \\
&\le 1-\nu(\pi_n^{-1} (\overline{\pi_n (K)}))
\le 1-\nu(K) <\nu(A)+\ep
\end{align*}
for any large $n$.
Thus $\limsn \nu_n (A)\le\nu(A)$  and the Portmanteau theorem implies that $\seqn{\nu_n}$ converges weakly to $\nu$.
\end{proof}

\begin{lem} \label{l:MI}
Let $\nu_0, \nu_1 \in \DomEnt{X}$ and $W_{2, d_{X}} (\nu_0, \nu_1)<\infty$. Define
$\nu_{n, i} :=(\pi_n)_* \nu_i \in \mathcal P(X_n)$ $(i=0, 1)$.
Then
\[
\Ent{\mu_{X_n}}{\nu_{n, i}} \le  \Ent{\mu_X}{\nu_i}  \quad\text{ and }\quad
\lim_{n \to \infty} W_{2, d_{X_n}} (\nu_{n, 0}, \nu_{n, 1}) = W_{2, d_X} (\nu_0, \nu_1) \fstop
\]
\end{lem}
\begin{proof}
The entropy contraction under push-forward  is standard (see, e.g., \cite[(7.4)]{AGS;C}). We prove the equality for the
Wasserstein distances.

\medskip
\noindent\textit{Disintegration and a Borel lifting of couplings.}
Since $(X,\tau_X)$ and $(X_n,\tau_{X_n})$ are Polish, $\pi_n:X\to X_n$ is Borel and
$(\pi_n)_*\mu_X=\mu_{X_n}$, there exists a disintegration kernel
\[
K_n:X_n\to \mathcal{P}(X),\qquad x_n\mapsto K_n(x_n),
\]
such that for every bounded Borel function $\psi:X\to\mathbb{R}$,
\begin{equation}\label{e:disint-mu-MI}
\int_X \psi(x)\,d\mu_X(x)
=\int_{X_n}\left(\int_X \psi(x)\,dK_n(x_n)(x)\right)d\mu_{X_n}(x_n).
\end{equation}
Moreover, $K_n(x_n)$ is concentrated on the fiber $\pi_n^{-1}(\{x_n\})$ for $\mu_{X_n}$-a.e.\ $x_n$. 

Fix a countable basis $(U_m)_{m\in\mathbb{N}}$ of $(X_n,\tau_{X_n})$ and define the countable $\pi$-system
\[
\mathcal{A}_n:=\Bigl\{\bigcap_{j=1}^r U_{m_j}:\ r\in\mathbb{N},\ m_1,\dots,m_r\in\mathbb{N}\Bigr\},
\qquad \sigma(\mathcal{A}_n)=\mathcal{B}(X_n).
\]
Define
\[
G_n:=\bigcap_{A\in\mathcal{A}_n}\Bigl\{x_n\in X_n:\ K_n(x_n)\bigl(\pi_n^{-1}(A)\bigr)=\mathbf{1}_{A}(x_n)\Bigr\}.
\]
For $A\in\mathcal{A}_n$, set $G_{n,A}:=\{x_n:\ K_n(x_n)(\pi_n^{-1}(A))=\mathbf{1}_A(x_n)\}$, thus $G_n= \cap_{A \in \mathcal {A}_n} G_{n,A}$.
If $K_n(x_n)$ is concentrated on $\pi_n^{-1}(\{x_n\})$, then
$K_n(x_n)(\pi_n^{-1}(A))=\mathbf{1}_A(x_n)$ for every Borel $A\subset X_n$.
Since $K_n(x_n)$ is concentrated on $\pi_n^{-1}(\{x_n\})$ for $\mu_{X_n}$-a.e.\ $x_n$, it follows that
$\mu_{X_n}(G_{n,A})=1$ for every $A\in\mathcal{A}_n$. As $\mathcal{A}_n$ is countable,
\[
\mu_{X_n}(G_n)=\mu_{X_n}\Bigl(\bigcap_{A\in\mathcal{A}_n}G_{n,A}\Bigr)=1.
\]
As a consequence, for every $x_n\in G_n$, 
\begin{equation}\label{e:fiber-MI}
(\pi_n)_*K_n(x_n)=\delta_{x_n},
\end{equation}
because $(\pi_n)_*K_n(x_n)$ and $\delta_{x_n}$ agree on the $\pi$-system $\mathcal{A}_n$, which is thus extended to $\mathcal{B}(X_n)$.

\smallskip
In the later argument, we record the following identity: for every bounded Borel $\psi:X\to\mathbb{R}$ and every
nonnegative Borel $g\in L^1(\mu_{X_n})$,
\begin{equation}\label{e:pullout-MI}
\int_X \psi(x)\,g(\pi_n(x))\,d\mu_X(x)
=\int_{X_n} g(x_n)\left(\int_X \psi(x)\,dK_n(x_n)(x)\right)d\mu_{X_n}(x_n).
\end{equation}
Indeed, apply \eqref{e:disint-mu-MI} to the bounded Borel functions
$\psi\cdot (g\wedge M)\circ\pi_n$ and let $M\to\infty$. Then, the dominated convergence can apply since
$|\psi(x)(g\wedge M)(\pi_n(x))|\le \|\psi\|_\infty\, g(\pi_n(x))$ and
$g\circ\pi_n\in L^1(\mu_X)$.

\medskip
\noindent\textit{Step 0.}
Since $\nu_i\ll \mu_X$ and  $\nu_{n,i}\ll \mu_{X_n}$, we have~$\nu_{n,i}=\rho_{n,i}\,\mu_{X_n}$ for some
$\rho_{n,i}\in L^1(\mu_{X_n})$. Choose Borel representatives of $\rho_{n,i}$ (without relabelling them) and 
define
\[
\nu_i^{(n)} := (\rho_{n,i}\circ\pi_n)\,\mu_X\in\mathcal{P}(X),\qquad i=0,1.
\]
By Lemma~\ref{l;weakconv}, $\nu_i^{(n)}\Rightarrow \nu_i$ weakly in $\mathcal{P}(X)$ for $i=0,1$.
For later use, we  define 
\[
d_n(x,y):=d_{X_n}(\pi_n(x),\pi_n(y)),\qquad (x,y\in X).
\]
Then $d_n$ is $(\tau_X\times\tau_X)$-continuous, $d_n\le d_{n+1}$, and $\sup_{n}d_n=d_X$.

\medskip
\noindent\textit{Step 1: lifting couplings and identifying costs.}
Let $\gamma_n\in \Cpl(\nu_{n,0},\nu_{n,1})$. Define a {\it lifted} Borel probability measure
$\Lambda_n(\gamma_n)\in\mathcal{P}(X\times X)$ by duality: for every bounded Borel
$\varphi:X\times X\to\mathbb{R}$,
\begin{equation}\label{e:lift-MI}
\int_{X\times X}\varphi(x,y)\,d\Lambda_n(\gamma_n)(x,y)
:=\int_{X_n\times X_n}\left(\int_{X\times X}\varphi(x,y)\,d(K_n(x_n)\otimes K_n(y_n))(x,y)\right)
d\gamma_n(x_n,y_n).
\end{equation}
Since $\mu_{X_n}(X_n\setminus G_n)=0$ and the marginals of $\gamma_n$ are $\nu_{n,0},\nu_{n,1}\ll\mu_{X_n}$,
we have $\gamma_n((X_n\setminus G_n)\times X_n)=\gamma_n(X_n\times (X_n\setminus G_n))=0$, hence $\gamma_n$
is supported on $G_n\times G_n$.

\smallskip
The lifting map~$\Lambda_n$ has the following properties:
\\
\noindent\emph{(i) Projection identity.}
For rectangles $A\times B$ with $A,B\in\mathcal{A}_n$, using \eqref{e:fiber-MI} and that $\gamma_n$ is supported on
$G_n\times G_n$,
\[
\Lambda_n(\gamma_n)\bigl(\pi_n^{-1}(A)\times \pi_n^{-1}(B)\bigr)
=\int_{X_n\times X_n}\mathbf{1}_A(x_n)\mathbf{1}_B(y_n)\,d\gamma_n(x_n,y_n)
=\gamma_n(A\times B).
\]
Since $\{A\times B:\ A,B\in\mathcal{A}_n\}$ is a $\pi$-system generating $\mathcal{B}(X_n\times X_n)$, we have
\begin{equation}\label{e:proj-MI}
(\pi_n^{\times 2})_*\Lambda_n(\gamma_n)=\gamma_n.
\end{equation}

\smallskip
\noindent\emph{(ii) Marginals.}
Let $E\in\mathcal{B}(X)$. Since the first marginal of $\gamma_n$ is $\nu_{n,0}$,
\[
\Lambda_n(\gamma_n)(E\times X)
=\int_{X_n} K_n(x_n)(E)\,d\nu_{n,0}(x_n)
=\int_{X_n}\rho_{n,0}(x_n)K_n(x_n)(E)\,d\mu_{X_n}(x_n).
\]
Applying \eqref{e:pullout-MI} with $\psi=\mathbf{1}_E$ and $g=\rho_{n,0}$ gives
\[
\Lambda_n(\gamma_n)(E\times X)=\int_E \rho_{n,0}(\pi_n(x))\,d\mu_X(x)=\nu_0^{(n)}(E).
\]
Similarly, the second marginal is $\nu_1^{(n)}$. Hence $\Lambda_n(\gamma_n)\in\Cpl(\nu_0^{(n)},\nu_1^{(n)})$.

\smallskip
\noindent\emph{(iii) Cost identity.}
By the definition of $d_n$, we have by \eqref{e:proj-MI} that 
\[
\int_{X\times X} d_n(x,y)^2\,d\Lambda_n(\gamma_n)(x,y)
=\int_{X_n\times X_n} d_{X_n}(x_n,y_n)^2\,d\gamma_n(x_n,y_n).
\]

Define
\[
W_{2,d_n}(\nu_0^{(n)},\nu_1^{(n)})^2 := \inf_{\Gamma\in\Pi(\nu_0^{(n)},\nu_1^{(n)})}\int_{X\times X} d_n^2\,d\Gamma,
\quad
W_{2,d_{X_n}}(\nu_{n,0},\nu_{n,1})^2 := \inf_{\tilde\Gamma\in\Pi(\nu_{n,0},\nu_{n,1})}\int_{X_n\times X_n} d_{X_n}^2\,d\tilde\Gamma.
\]
Then the preceding identities yield
\begin{equation}\label{e:W-ident-MI}
W_{2,d_{X_n}}(\nu_{n,0},\nu_{n,1}) = W_{2,d_n}(\nu_0^{(n)},\nu_1^{(n)}).
\end{equation}
Indeed, the above lifting arguments give $W_{2,d_n}\le W_{2,d_{X_n}}$:
Let $\tilde\Gamma\in\Cpl(\nu_{n,0},\nu_{n,1})$ be arbitrary and let $\Lambda_n(\tilde\Gamma)$ be its lift defined by
\eqref{e:lift-MI}. By Step~1(ii) we have $\Lambda_n(\tilde\Gamma)\in\Cpl(\nu_0^{(n)},\nu_1^{(n)})$, and by Step~1(iii),
\[
\int_{X\times X} d_n(x,y)^2\,d\Lambda_n(\tilde\Gamma)(x,y)
=\int_{X_n\times X_n} d_{X_n}(x_n,y_n)^2\,d\tilde\Gamma(x_n,y_n).
\]
Therefore, taking the infimum over $\Gamma\in\Cpl(\nu_0^{(n)},\nu_1^{(n)})$, we obtain
\[
W_{2,d_n}(\nu_0^{(n)},\nu_1^{(n)})^2
\le \int_{X\times X} d_n^2\,d\Lambda_n(\tilde\Gamma)
=\int_{X_n\times X_n} d_{X_n}^2\,d\tilde\Gamma.
\]
Finally, taking the infimum over $\tilde\Gamma\in\Cpl(\nu_{n,0},\nu_{n,1})$, we obtain
\[
W_{2,d_n}(\nu_0^{(n)},\nu_1^{(n)})^2 \le W_{2,d_{X_n}}(\nu_{n,0},\nu_{n,1})^2,
\]
hence $W_{2,d_n}(\nu_0^{(n)},\nu_1^{(n)}) \le W_{2,d_{X_n}}(\nu_{n,0},\nu_{n,1})$.

Conversely, if $\Gamma\in\Cpl(\nu_0^{(n)},\nu_1^{(n)})$, then
$\tilde\Gamma:=(\pi_n^{\times 2})_*\Gamma\in\Cpl(\nu_{n,0},\nu_{n,1})$ and
$\int d_{X_n}^2\,d\tilde\Gamma=\int d_n^2\,d\Gamma$, hence $W_{2,d_{X_n}}\le W_{2,d_n}$.

\medskip
\noindent\textit{Step 2: the $\liminf$ inequality.}
Since $W_{2,d_X}(\nu_0,\nu_1)<\infty$, we can choose $\bar\Gamma\in\Cpl(\nu_0,\nu_1)$ with
\[
\int_{X\times X} d_X(x,y)^2\,d\bar\Gamma(x,y)<\infty.
\]
Define $\bar\gamma_n:=(\pi_n^{\times 2})_*\bar\Gamma\in\Cpl(\nu_{n,0},\nu_{n,1})$ and set
\[
\widetilde\Gamma_n:=\Lambda_n(\bar\gamma_n)\in\mathcal P(X\times X),
\]
where $\Lambda_n$ is the lifting map constructed in Step~1. By Step~1(ii), $\widetilde\Gamma_n$ is a coupling of
$\nu_0^{(n)}$ and $\nu_1^{(n)}$.
Moreover, by Step~1(iii),
\[
\int_{X\times X} d_n(x,y)^2\,d\widetilde\Gamma_n(x,y)
=\int_{X_n\times X_n} d_{X_n}(x_n,y_n)^2\,d\bar\gamma_n(x_n,y_n)
=\int_{X\times X} d_n(x,y)^2\,d\bar\Gamma(x,y).
\]
Since $d_n\le d_X$, we have $\int d_n^2\,d\bar\Gamma\le \int d_X^2\,d\bar\Gamma<\infty$, hence
\[
W_{2,d_n}(\nu_0^{(n)},\nu_1^{(n)})^2
\le \int_{X\times X} d_n^2\,d\widetilde\Gamma_n
<\infty.
\]
%

Take an
optimal coupling $\Gamma_n\in\Cpl(\nu_0^{(n)},\nu_1^{(n)})$, viz., 
\[
\int_{X\times X} d_n^2\,d\Gamma_n = W_{2,d_n}(\nu_0^{(n)},\nu_1^{(n)})^2.
\]
The families $(\nu_0^{(n)})_n$ and $(\nu_1^{(n)})_n$ are tight since $\nu_i^{(n)}\Rightarrow \nu_i$, hence
$(\Gamma_n)_n$ is tight on $X\times X$. By the Prokhorov theorem, up to a subsequence, $\Gamma_n\Rightarrow \Gamma$ weakly for
some $\Gamma\in\mathcal{P}(X\times X)$. Passing to the limit in the marginals, we have $\Gamma\in\Cpl(\nu_0,\nu_1)$.

We claim
\[
\liminf_{n\to\infty}\int d_n^2\,d\Gamma_n \ge \int d_X^2\,d\Gamma.
\]
Fix $m\in\mathbb{N}$. Since $d_n^2\ge d_m^2$ for all $n\ge m$,
\[
\liminf_{n\to\infty}\int d_n^2\,d\Gamma_n \ge \liminf_{n\to\infty}\int d_m^2\,d\Gamma_n.
\]
As $d_m^2$ is continuous on $X\times X$, the Portmanteau theorem yields
\[
\liminf_{n\to\infty}\int d_m^2\,d\Gamma_n \ge \int d_m^2\,d\Gamma.
\]
Hence $\liminf_n\int d_n^2\,d\Gamma_n \ge \int d_m^2\,d\Gamma$ for every $m$. Taking $\sup_m$ gives
\[
\liminf_{n\to\infty}\int d_n^2\,d\Gamma_n \ge \sup_{m}\int d_m^2\,d\Gamma
= \int \sup_m d_m^2\,d\Gamma = \int d_X^2\,d\Gamma,
\]
where the last two equalities use monotone convergence and $\sup_m d_m^2=d_X^2$.
Therefore,
\[
\liminf_{n\to\infty} W_{2,d_n}(\nu_0^{(n)},\nu_1^{(n)})^2
=\liminf_{n\to\infty}\int d_n^2\,d\Gamma_n
\ge \int d_X^2\,d\Gamma
\ge W_{2,d_X}(\nu_0,\nu_1)^2,
\]
hence
\[
\liminf_{n\to\infty} W_{2,d_n}(\nu_0^{(n)},\nu_1^{(n)}) \ge W_{2,d_X}(\nu_0,\nu_1).
\]
Using \eqref{e:W-ident-MI}, we obtain
\begin{align} \label{e: WWE}
\liminf_{n\to\infty} W_{2,d_{X_n}}(\nu_{n,0},\nu_{n,1}) = \liminf_{n\to\infty} W_{2,d_n}(\nu_0^{(n)},\nu_1^{(n)})  \ge W_{2,d_X}(\nu_0,\nu_1).
\end{align}

\medskip
\noindent\textit{Step 3: the $\limsup$ inequality.}
Fix $\varepsilon>0$. Choose $\bar\Gamma\in\Cpl(\nu_0,\nu_1)$ such that
\[
\int_{X\times X} d_X^2\,d\bar\Gamma \le W_{2,d_X}(\nu_0,\nu_1)^2+\varepsilon.
\]
Let $\bar\Gamma_n:=(\pi_n^{\times 2})_*\bar\Gamma\in\Cpl(\nu_{n,0},\nu_{n,1})$. Then
\[
W_{2,d_{X_n}}(\nu_{n,0},\nu_{n,1})^2
\le \int_{X_n\times X_n} d_{X_n}^2\,d\bar\Gamma_n
= \int_{X\times X} d_n^2\,d\bar\Gamma
\le \int_{X\times X} d_X^2\,d\bar\Gamma
\le W_{2,d_X}(\nu_0,\nu_1)^2+\varepsilon.
\]
Taking $\limsup_{n\to\infty}$ and then $\varepsilon\downarrow 0$, we obtain
\[
\limsup_{n\to\infty} W_{2,d_{X_n}}(\nu_{n,0},\nu_{n,1}) \le W_{2,d_X}(\nu_0,\nu_1).
\]

Combining the $\liminf$ and $\limsup$ inequalities, we conclude
\[
\lim_{n\to\infty} W_{2,d_{X_n}}(\nu_{n,0},\nu_{n,1}) = W_{2,d_X}(\nu_0,\nu_1). \qedhere
\]
\end{proof}

\begin{proof}[Proof of Theorem \ref{t:PCD}]
Let \(\nu_0,\nu_1\in D(\operatorname{Ent}_{\mu_X})\) satisfy
\(
W_{2,d_X}(\nu_0,\nu_1)<\infty .
\)
For \(i=0,1\), set
\(
\nu_{n,i}:=(\pi_n)_*\nu_i .
\)
By Lemma~\ref{l:MI}, we have
\[
\operatorname{Ent}_{\mu_{X_n}}(\nu_{n,i})
\le
\operatorname{Ent}_{\mu_X}(\nu_i)
\]
and
\[
W_{2,d_{X_n}}(\nu_{n,0},\nu_{n,1})
\longrightarrow
W_{2,d_X}(\nu_0,\nu_1).
\]
Put
\[
L:=W_{2,d_X}(\nu_0,\nu_1).
\]

Since \(X_n\) satisfies \(\CD(K,\infty)\), for each \(n\in\mathbb N\) there exists a
constant-speed \(W_{2,d_{X_n}}\)-geodesic
\(
(\nu_{n,t})_{t\in[0,1]}\subset D(\operatorname{Ent}_{\mu_{X_n}})
\)
joining \(\nu_{n,0}\) and \(\nu_{n,1}\), along which
\[
\operatorname{Ent}_{\mu_{X_n}}(\nu_{n,t})
\le
(1-t)\operatorname{Ent}_{\mu_{X_n}}(\nu_{n,0})
+t\operatorname{Ent}_{\mu_{X_n}}(\nu_{n,1})
-\frac K2 t(1-t)
W_{2,d_{X_n}}(\nu_{n,0},\nu_{n,1})^2
\]
for every \(t\in[0,1]\). Write
\[
\nu_{n,t}=\rho_{n,t}\mu_{X_n}
\]
and define the lifted probability measure on \(X\) by
\(
\nu_t^{(n)}:=(\rho_{n,t}\circ\pi_n)\mu_X .
\)
Then
\[
(\pi_n)_*\nu_t^{(n)}=\nu_{n,t}
\]
and
\(
\operatorname{Ent}_{\mu_X}(\nu_t^{(n)})
=
\operatorname{Ent}_{\mu_{X_n}}(\nu_{n,t}).
\)

Let
\(
Q:=\mathbb Q\cap[0,1].
\)
For every \(q\in Q\), the above entropy estimate gives
\[
\sup_n \operatorname{Ent}_{\mu_X}(\nu_q^{(n)})<\infty .
\]
Indeed,
\[
\begin{aligned}
\operatorname{Ent}_{\mu_X}(\nu_q^{(n)})
&=
\operatorname{Ent}_{\mu_{X_n}}(\nu_{n,q})  \\
&\le
(1-q)\operatorname{Ent}_{\mu_X}(\nu_0)
+q\operatorname{Ent}_{\mu_X}(\nu_1)
-\frac K2 q(1-q)W_{2,d_{X_n}}(\nu_{n,0},\nu_{n,1})^2,
\end{aligned}
\]
and \(W_{2,d_{X_n}}(\nu_{n,0},\nu_{n,1})\to L\). Hence \(\{\nu_q^{(n)}\}_n\) is tight for each
\(q\in Q\). Since \(Q\) is countable, by the standard diagonal argument, we can extract  a single (non-relabelled) subsequence and measures \(\nu_q\in \mathcal P(X)\), \(q\in Q\), such that
\[
\nu_q^{(n)}\rightharpoonup \nu_q
\qquad\text{for every }q\in Q.
\]
For \(q=0,1\), Lemma~\ref{l;weakconv} gives the endpoint measures \(\nu_0\) and
\(\nu_1\).

We claim that \((\nu_q)_{q\in Q}\) is a constant-speed geodesic on rational times. Let
\(q,r\in Q\). By the lower semicontinuity argument in Lemma~\ref{l:MI}, and by the lifted
Wasserstein identity \eqref{e:W-ident-MI} applied to the pair \(\nu_{n,q},\nu_{n,r}\), we get
\[
\begin{aligned}
W_{2,d_X}(\nu_q,\nu_r)
&\le
\liminf_{n\to\infty} W_{2,d_n}(\nu_q^{(n)},\nu_r^{(n)})  \\
&=
\liminf_{n\to\infty} W_{2,d_{X_n}}(\nu_{n,q},\nu_{n,r})  \\
&=
|q-r|\lim_{n\to\infty}W_{2,d_{X_n}}(\nu_{n,0},\nu_{n,1})  \\
&=
|q-r|L .
\end{aligned}
\]
For \(0\le q\le r\le1\), applying this estimate to the pairs \((0,q)\), \((q,r)\), and
\((r,1)\), and using the triangle inequality, we obtain
\[
\begin{aligned}
L
&=
W_{2,d_X}(\nu_0,\nu_1) \\
&\le
W_{2,d_X}(\nu_0,\nu_q)
+
W_{2,d_X}(\nu_q,\nu_r)
+
W_{2,d_X}(\nu_r,\nu_1) \\
&\le
qL+(r-q)L+(1-r)L
=
L.
\end{aligned}
\]
Thus the equality holds. Therefore
\[
W_{2,d_X}(\nu_q,\nu_r)=|q-r|L
\qquad
\forall q,r\in Q.
\]

We next pass the entropy inequality to rational times. By the lower semicontinuity of the entropy,
\begin{align} \label{e:EBR}
\operatorname{Ent}_{\mu_X}(\nu_q) \notag
&\le
\liminf_{n\to\infty}
\operatorname{Ent}_{\mu_X}(\nu_q^{(n)})  \\ \notag
&=
\liminf_{n\to\infty}
\operatorname{Ent}_{\mu_{X_n}}(\nu_{n,q})  \\ \notag
&\le
\limsup_{n\to\infty}
\left[
(1-q)\operatorname{Ent}_{\mu_{X_n}}(\nu_{n,0})
+q\operatorname{Ent}_{\mu_{X_n}}(\nu_{n,1})
-\frac K2 q(1-q)
W_{2,d_{X_n}}(\nu_{n,0},\nu_{n,1})^2
\right]  \\
&\le
(1-q)\operatorname{Ent}_{\mu_X}(\nu_0)
+q\operatorname{Ent}_{\mu_X}(\nu_1)
-\frac K2 q(1-q)L^2 .
\end{align}

It remains to define \(\nu_t\) for arbitrary \(t\in[0,1]\). Let \(t\in[0,1]\), and let
\((q_j)_j\subset Q\) be any sequence with \(q_j\to t\). Then, we have 
\[
W_{2,d_X}(\nu_{q_j},\nu_{q_k})=|q_j-q_k|L.
\]
In particular, \((\nu_{q_j})_j\) is \(W_{2,d_X}\)-Cauchy. Since the entropy bounds above are
uniform for \(q_j\in[0,1]\), the sequence is tight. Choose a weak limit point, denoted  by
\(\nu_t\). This definition is independent of the approximating rational sequence and of the chosen
weakly convergent subsequence: if \(\eta\) and \(\eta'\) are two such limits, then the weak joint-lower
semicontinuity implies
\[
W_{2,d_X}(\eta,\eta')=0,
\]
hence \(\eta=\eta'\).
Moreover, for \(q\in Q\),  by choosing rational \(q_j\to t\), and using the lower semicontinuity, we have 
\[
W_{2,d_X}(\nu_q,\nu_t)
\le
\liminf_{j\to\infty} W_{2,d_X}(\nu_q,\nu_{q_j})
=
|q-t|L.
\]
Using the triangle inequality with the endpoints as above, we obtain the equality. Hence
\[
W_{2,d_X}(\nu_s,\nu_t)=|s-t|L
\qquad
\forall s,t\in[0,1].
\]
Thus \((\nu_t)_{t\in[0,1]}\) is a constant-speed \(W_{2,d_X}\)-geodesic connecting
\(\nu_0\) and \(\nu_1\).

Finally, let \(t\in[0,1]\) and choose \(q_j\in Q\) with \(q_j\to t\). Since
\(\nu_{q_j}\rightharpoonup\nu_t\), the lower semicontinuity of the entropy and \eqref{e:EBR} conclude
\[
\operatorname{Ent}_{\mu_X}(\nu_t)
\le
(1-t)\operatorname{Ent}_{\mu_X}(\nu_0)
+t\operatorname{Ent}_{\mu_X}(\nu_1)
-\frac K2 t(1-t)W_{2,d_X}(\nu_0,\nu_1)^2 .
\]
Therefore \(X\) satisfies \(\CD(K,\infty)\).
\end{proof}

\section{Mosco Convergence} \label{sec:Mosco}
In this section, we prove the Mosco convergence of Cheeger energies in inverse systems. 
Let ${\rm Lip}(d_X)$ denote the space of Lipschitz functions $u: X \to \R$ and ${\rm Lip}_b(d_X)$ denote the subspace consisting of bounded functions. Since $d_X$ is merely an extended distance,  elements both in ${\rm Lip}(d_X)$ and  ${\rm Lip}_b(d_X)$ are not necessarily measurable in terms of the Borel $\sigma$-algebra $\mathcal B(X, \tau_X)$ with respect to $\tau_X$. We, therefore, introduce a further notation: for a $\sigma$-algebra $\Sigma$ on $X$, define
\begin{align*}
{\rm Lip}(d_X, \Sigma)&:=\{u \in {\rm Lip}(d_X): u\ \text{is $\Sigma$-measurable} \}, 
\\
{\rm Lip}_b(d_X, \Sigma)&:=\{u \in {\rm Lip}_b(d_X): u\ \text{is $\Sigma$-measurable}\} \fstop
\end{align*}
We recall the definition of the Cheeger energy on a Polish extended metric measure space based on the $L^2$-relaxation of the pointwise slope. 
\begin{defn}[Cheeger energy: {\cite[Theorem 4.5]{AGS;C}}] \label{d:Ch}
Let $X=(X, \tau_X, d_X, \mu_X)$ be a Polish extended mm-space. The Cheeger energy $\Ch_{d_X, \mu_X}: L^2(X, \mu_X) \to [0, \infty]$ is defined as 
\begin{align*}
 \Ch_{d_X, \mu_X}(u)
:=\ \inf\biggl\{ \liminf_{n \to \infty} \int_X |\SL{d_X}{u_n}|^2 \diff \mu_X : u_n \xrightarrow{L^2(\mu_X)} u, \ u_n \in {\rm Lip}_b(d_X, \mathcal B(X))\cap L^2(\mu_X) \biggr\} \comma
\end{align*}
where {\it the $d_X$-slope $|\SL{d_X}{u}|$} is defined as  
\begin{align} \label{e:SL}
|\SL{d_X}{u}|(x):=
\begin{cases} \displaystyle
\limsup_{y \to x}\frac{|u(y)-u(x)|}{d_X(y, x)} \quad &\text{when $x$ is not isolated};
\\
0 \quad &\text{otherwise.}
\end{cases}
\end{align}  
We define $(1,2)$-Sobolev space by
\[W^{1,2}(X)=W^{1,2}(X, d_X, \mu_X):=\{u \in L^2(\mu_X): \Ch_{d_X, \mu_X}(u)<\infty\}.\]
We define $\Ch_{d_X, \mu_X}(u)=+\infty$ for $u \in L^2(X, \mu_X) \setminus W^{1,2}(X, d_X, \mu_X)$.
If the following parallelogram law holds
$$\Ch_{d_X, \mu_X}(u+v) + \Ch_{d_X, \mu_X}(u-v) = 2\Ch_{d_X, \mu_X}(u) + 2\Ch_{d_X, \mu_X}(v) \quad u, v \in W^{1,2}(X) $$
we say that $X$ is {\it infinitesimally Hilbertian}, which was introduced in \cite{Gi}. 
\end{defn}

\begin{defn}[$L^2$-gradient flow {\cite[Section 4.2]{AGS;C}}]
We denote by $(H_t^X)_{t \ge 0}$ the $L^2$-gradient flow of $\Ch_{d_X, \mu_X}$, namely,  for all $f_0 \in L^2(X, \mu_X)$, the map $t \mapsto f_t:=H_t^X f_0$ from $(0, \infty)$ to $L^2(X, \mu_X)$ is the unique locally Lipschitz curve with $f_t \to f_0$ as $t \downarrow 0$ whose derivative satisfies 
$$\frac{d}{dt} f_t \in \partial ^-\Ch_{d_X, \mu_X}(f_t) \quad \text{for a.e.~$t \in (0,\infty)$} .$$
Here, for $f \in W^{1,2}(X, d_X, \mu_X)$, we define $$l \in \partial ^-\Ch_{d_X, \mu_X}(f)  \overset{def}\iff \int_X l(g-f) d \mu_X \le \Ch_{d_X, \mu_X}(g)- \Ch_{d_X, \mu_X}(f)  \qquad \text{for every $g \in L^2(X, \mu_X)$}.$$
\end{defn}
  Let $\seqn{X_n}$ be a sequence of mm-spaces with $X_n \prec X_{n+1}$ for all $n$ and $X:=\varprojlim X_n$ its inverse limit mm-space. Let $d_n:X\times X\to[0,\infty)$ be the semidistance
\[
d_n(x,y):=d_{X_n}(\pi_n(x),\pi_n(y)).
\]
Let $X_n^{\mathrm{q}}:=X/{\sim}$ be the quotient set induced by $x\sim y\iff d_n(x,y)=0$, and let
$\mathrm{p}_n:X\to X_n^{\mathrm{q}}$ be the canonical projection. Define the quotient distance
$\tilde d_n([x],[y]):=d_n(x,y)$ and the push-forward measure $m_n:=(\mathrm{p}_n)_*\mu_X$.
Denote by $\Ch_{\tilde d_n,m_n}$ the Cheeger energy on 
$(X_n^{\mathrm{q}},\tilde d_n,m_n)$.
\medskip
\begin{defn}
For $f\in L^2(X,\mu_X)$ we define the {\it lifted Cheeger energy} $\Ch_{d_n,\mu_X}(f)$ by
\[
\Ch_{d_n,\mu_X}(f):=
\begin{cases}
\Ch_{\tilde d_n,m_n}(g) & \text{if } f=g\circ \mathrm{p}_n \text{ for some } g\in L^2(X_n^{\mathrm{q}},m_n),\\
+\infty & \text{otherwise,}
\end{cases}
\]
and 
\begin{align*}
W^{1,2}(X,d_n,\mu_X)
:=\{f\circ \mathrm{p}_n:\ f\in W^{1,2}(X_n^\mathrm q ,\tilde d_{n},m_{n})\}.
\end{align*}
Let $(\tilde H^n_t)_{t \ge 0}$ and $(H^{(n)}_t)_{t \ge 0}$ be the $L^2$-gradient flow of $\Ch_{\tilde d_n,m_n}$ and $\Ch_{d_n, \mu_X}$ respectively. Then,  
\[
H^{(n)}_t f:=(\tilde H^n_t g) \circ  \mathrm{p}_n \quad  \text{for} \quad f=g\circ \mathrm{p}_n \text{ for some } g\in L^2(X_n^{\mathrm{q}},m_n).
\]
\end{defn}
\medskip
\begin{lem}\label{l:M=M}
The following hold:
\begin{align}
\Ch_{d_n,\mu_X}(f\circ\pi_n)&=\Ch_{d_{X_n},\mu_{X_n}}(f), \qquad f\in L^2(X_n,\mu_{X_n})
\\
W^{1,2}(X, d_n,\mu_X)
&=\{f\circ\pi_n:\ f\in W^{1,2}(X_n,d_{X_n},\mu_{X_n})\}  \notag 
\\
 \label{d:CEC} H^{(n)}_tf&= (H^{X_n}_t g) \circ \pi_n \quad  \text{for} \quad f=g\circ \pi_n \text{ for some } g\in L^2(X_n, \mu_{X_n}). \notag
\end{align}
\end{lem}
\begin{proof}
Since
\(
d_n(x,y)=0 \Longleftrightarrow \pi_n(x)=\pi_n(y),
\)
the map
\[
J_n:X_n^q\to X_n,\qquad J_n([x])=\pi_n(x),
\]
is a well-defined isometric embedding, and
\(
J_n\circ \mathrm p_n=\pi_n.
\)
Its image is \(\pi_n(X)\subset X_n\).
Thus,  $\tilde d_n$ is identical with $d_{X_n}$, and $m_n=(\pi_n)_*\mu_X=\mu_{X_n}$.
Note that $\pi_n(X) \subset X_n$ is dense and a set of full measure with respect to $\mu_{X_n}$. In particular the metric completion of $\pi_n(X)$ is $X_n$.
Thus, the metric completion of \(X_n^q\) is identified with~\(X_n\). By the
invariance of the Cheeger energy with respect to restriction and completion
\cite[Section~3.1.2]{Sa}, the Cheeger energy on \(X_n^q\) is identified with the Cheeger energy on
\((X_n,d_{X_n},\mu_{X_n})\).
Thus, for every $g\in L^2(X_n,\mu_{X_n})$,
\[
\Ch_{d_n,\mu_X}(g\circ\pi_n)=\Ch_{d_{X_n},\mu_{X_n}}(g),
\]
and $\Ch_{d_n,\mu_X}(f)=+\infty$ whenever $f$ is not  in $\mathcal B_n(X):=\pi_n^{-1}\mathcal B(X_n)$-measurable. In particular, this concludes~the other two equalities.
\end{proof}

\begin{defn}[Mosco convergence] \label{d:Mosco}
Let $X$ be a Polish extended mm-space and $\seqn{(\E_n, \F_n)}$ a sequence of $L^2(\mu_X)$-lower semicontinuous convex functionals $\E_n$ with domain $\F_n \subset L^2(\mu_X)$. 
We say that $\E_n$ \emph{Mosco converges} to $\E$ if 
\begin{enumerate}
\item \label{c:M1} for any $u_n \in \mathcal F_n$ converging weakly to $u$ in $L^2(\mu_X)$, 
$$\liminf_{n \to \infty}\E_n(u_n) \ge \E(u)\ ;$$
\item  \label{c:M2} for any $u \in \F$, there exists $u_n \in \F_n$ so that $u_n$ converges strongly to $u$ in $L^2(\mu_X)$ and 
$$\limsup_{n \to \infty}\E_n(u_n) \le \E(u) \fstop$$
\end{enumerate}
\end{defn}

\smallskip
The main theorem of this section is the following:
\begin{thm}[Mosco convergence]\label{t:Mosco}
Let $\seqn{X_n}$ be a sequence of mm-spaces with $X_n \prec X_{n+1}$ for all $n$
and $X:=\varprojlim X_n$ its inverse limit extended mm-space.
Then, $\Ch_{d_n, \mu_X}$ converges to $\Ch_{d_X, \mu_X}$ in the Mosco sense.
\end{thm}
\begin{proof}
For simplicity set 
\[
\Ch_n:=\Ch_{d_n,\mu_X},\qquad \Ch:=\Ch_{d_X,\mu_X}.
\]
Since $d_n \uparrow d_X$ pointwise and $d_n\le d_{n+1}$, we have the monotonicity
$$|\SL{d_X}{u}| \le |\SL{d_{n+1}}{u}| \le |\SL{d_n}{u}|.$$
Therefore, 
\begin{equation}\label{e:Ch-mon}
\Ch\le \Ch_{n+1} \le \Ch_n\qquad  n\in\N.
\end{equation}

\smallskip
\noindent\emph{(M1).}
Let $u_n\in L^2(\mu_X)$ converge weakly to $u$ in $L^2(\mu_X)$. 
Since $\Ch$ is convex and lower semicontinuous in $L^2(\mu_X)$, it is also weakly lower semicontinuous, so
\[
\Ch(u)\le \liminf_{n\to\infty}\Ch(u_n).
\]
Using \eqref{e:Ch-mon} we obtain $\Ch(u_n)\le \Ch_n(u_n)$ for every $n$, hence
\[
\Ch(u)\le \liminf_{n\to\infty}\Ch(u_n)\le \liminf_{n\to\infty}\Ch_n(u_n),
\]
which is exactly (M1).

\smallskip
\noindent\emph{(M2).}
Fix $u\in L^2(\mu_X)$ with $\Ch(u)<\infty$.
By the last statement of \cite[Theorem~9.2]{AES}, applied to the monotone family $(d_n)_n$, there exist
a strictly increasing sequence $(n_k)_k$ and bounded $d_{n_k}$-Lipschitz functions $u_k\in L^2(\mu_X)$ such that
\[
u_k\to u\ \text{in }L^2(\mu_X)
\qquad\text{and}\qquad
\Ch_{n_k}(u_k)\to \Ch(u).
\]
Define a full sequence $(v_n)_n$ by setting $v_n:=u_k$ whenever $n_k\le n<n_{k+1}$.
Then $v_n\to u$ strongly in $L^2(\mu_X)$, and since $\Ch_n$ is monotone decreasing in $n$ (because $d_n$ is
monotone increasing), we have for $n_k\le n<n_{k+1}$:
\[
\Ch_n(v_n)=\Ch_n(u_k)\le \Ch_{n_k}(u_k).
\]
Therefore,
\[
\limsup_{n\to\infty}\Ch_n(v_n)\le \lim_{k\to\infty}\Ch_{n_k}(u_k)=\Ch(u),
\]
which is (M2).
\end{proof}

\begin{rem}\normalfont
In~\cite{LP}, the Mosco convergence of Cheeger energies has been proven in the setting that  distance functions $d_{X_n}$ are monotone increasing to a {\it non-extended} distance $d_{X}$ and the topology generated by $d_{X_n}$ and $d_X$ are the same. 
\end{rem}
\begin{cor} \label{c:IH}
Suppose the same assumption as in Theorem \ref{t:Mosco} and furthermore that $\Ch_{d_{X_n}, \mu_{X_n}}$ is infinitesimally Hilbertian for any $n \in \N$. 
Then, $\Ch_{d_X, \mu_X}$ is infinitesimally Hilbertian. 
\end{cor}
\begin{proof}
It follows from Theorem~\ref{t:Mosco} with the same argument in~\cite[Thm.~7.2]{GMS} .
\end{proof}
By the combination of Theorem~\ref{t:PCD} and Corollary~\ref{c:IH}, we obtain the stability of {\it Riemannian} curvature-dimension condition $\RCD(K,\infty)$ in the inverse system. We say that a Polish extended metric measure space $X$ satisfies {\it $\RCD(K,\infty)$ condition} if $X$ satisfies $\CD(K,\infty)$ condition and $X$ is infinitesimally Hilbertian. 
\begin{cor} \label{c:RCD}
Suppose that $X_n \prec X_{n+1}$ and $X_n$ satisfies $\RCD(K,\infty)$ for every $n \in \N$. Then, the inverse limit $X:=\varprojlim X_n$ satisfies $\RCD(K,\infty)$.
\end{cor}

The following corollary is a direct consequence of \cite[Theorem~9.2]{AES}.
\begin{cor} \label{c:SR}
Suppose the same assumption as in Theorem~\ref{t:Mosco} and let $\{H_t^{(n)}\}_{t \ge 0}$ and $\{H_t^{X
}\}_{t \ge 0}$ be the $L^2$-gradient flows of $\Ch_{d_{n}, \mu_{X}}$ and $\Ch_{d_{X}, \mu_{X}}$ respectively.
Then, for any $f \in L^2(X, \mu_{X})$ and any $t>0$
$$H_t^{(n)}f \to H^X_t f \quad  \text{$L^2(X, \mu_X)$-strongly} \fstop$$
\end{cor}

\section{Functional Inequalities} \label{s:FI} 
In this section, we prove the stability of several functional inequalities under  the inverse limit.
We use the shorthand notation $\mu_X(u)=\int_X u \diff \mu_X$ for $u \in L^1(\mu_X)$.
\begin{defn} \label{d:FI}
Let $X=(X, \tau_X, d_X, \mu_X)$ be a Polish extended mm-space and $\Ch=\Ch_{d_X, \mu_X}$ be the Cheeger energy on $X$.
\begin{enumerate}
\item We say that $X$ supports {\it the Poincar\'e inequality $\PI(C_P)$} if there exists a constant $C_P>0$ so that 
$$ \mu_X(|u-\mu_X(u)|) \le C_P \Ch(u)^{1/2} \quad \forall u \in \Lip_b(d_X, \mathcal{B}(X)) \fstop$$
\item We say that $X$ supports {\it the logarithmic Sobolev inequality $\LS(C_{LS}, D_{LS})$} if there exist constants $C_{LS}, D_{LS}>0$ so that 
$$\Ent{\mu_X}{u^2} \le C_{LS} \Ch(u) +  D_{LS}\mu_X(u^2) \quad \forall u \in W^{1,2}(X) \comma$$
where $\Ent{\mu_X}{u^2}:=\int_X u^2 \log u^2 \diff \mu_X$.
\item We say that $X$ supports {\it the Talagrand inequality $\TL(C_{T})$} if 
\begin{align*}
W_2(\mu_X, \sigma)^2 \le C_{T}\ENT_{\mu_X}(\sigma) \quad \forall \sigma \in \mathcal{P}_*(X) \fstop
\end{align*} 
\end{enumerate}
\end{defn}

Let $\{H_t\}_{t \ge 0}=\{H_t^{X}\}_{t \ge 0}$ be the $L^2$-gradient flow associated with~$\Ch=\Ch_{d_X, \mu_X}$.
We recall functional inequalities related to the Ricci curvature lower bound. 
\begin{defn} \label{d:CFI}
Let $X=(X, \tau_X, d_X, \mu_X)$ be a Polish extended mm-space and $K \in \R$, $\alpha \in (1, \infty)$.
The constant $K/(1-e^{-2Kt})$ is regarded as $1/2t$ if $K=0$.
\begin{enumerate}
\item We say that $X$ supports {\it the log-Harnack inequality $\LH(K)$} if for every non-negative $u \in L^\infty (\mu_X)$ and $t>0$, there exists $\Omega \subset X$ with $\mu_X(\Omega)=1$ such that 
$$H_t(\log u)(x) \le \log (H_tu)(y) + \frac{K d_X(x, y)^2}{2(1-e^{-2Kt})} \quad \text{$x, y \in \Omega$}\fstop$$

\item We say that $X$ supports {\it the dimension-free Harnack inequality $\DFH(K)$} if for every non-negative $u \in L^\infty (\mu_X)$ and $t>0$,there exists $\Omega \subset X$ with $\mu_X(\Omega)=1$ such that 
$$(H_tu)^\alpha(y)\le H_tu^\alpha(x) \exp\Bigl\{ \frac{\alpha K d_X(x, y)^2}{2(\alpha-1)(1-e^{-2Kt})}\Bigr\} \quad \text{$x, y \in \Omega$}\fstop$$

%
%

\item We say that $X$ supports {\it the Evolution Variational Inequality $\EVI(K,\infty)$} if for any $\sigma=\rho\cdot \mu_X \in D(\ENT_{\mu_X})$, $\nu \in D(\ENT_{\mu_X})$ with $W_{2, d_X}(\sigma, \nu)<\infty$, the following inequality holds for all $t>0$:
$$\frac{1}{2}\frac{d^+}{dt}W_{2, d_X}\bigl((H_t \rho) \cdot \mu_X, \nu \bigr)^2 + \frac{K}{2} W_{2, d_X}\bigl((H_t \rho) \cdot \mu_X, \nu\bigr)^2 \le \Ent{\mu_X}{\nu} - \Ent{\mu_X}{(H_t \rho) \cdot \mu_X}\comma$$
where $\frac{d^+}{dt}$ stands for the upper right derivative in $t$. 
\end{enumerate}
\end{defn}

\subsection{Stability for $\LS, \PI, \LH, \DFH$}

\begin{thm}[log-Sobolev, Poincar\'e inequalities]\label{t:P}
Let $\seqn{X_n}$ be a sequence of mm-spaces with $X_n \prec X_{n+1}$ and  suppose that $\seqn{X_n}$ satisfies  $\LS(C_{1, n}, D_{1, n})$ $($resp.\  $\PI(C_{2, n}))$ with $C_1:=\lim_{n\to \infty}C_{1,n}<\infty$, $D_1:=\lim_{n\to \infty}D_{1,n}<\infty$ $($resp. $C_2:=\lim_{n \to \infty}C_{2,n}<\infty)$. Then the inverse limit $X$ satisfies $\LS(C_{1}, D_1)$ $($resp.\  $\PI(C_{2}))$.
\end{thm}
\begin{proof}
By hypothesis, we have 
$${\rm Ent}_n(u^2) \le C_{1, n}{\sf Ch}_{d_{X_n}, \mu_{X_n}}(u) + D_{1, n} \|u\|_{L^2(\mu_{X_n})}^2 \quad u \in W^{1,2}(X_n, d_{X_n}, \mu_{X_n})\fstop$$
By Lemma~\ref{l:M=M}, 
$${\rm Ent}(u^2\circ\pi_n) \le C_{1, n}{\sf Ch}_{d_n, \mu}(u\circ \pi_n) + D_{1, n} \|u\circ \pi_n\|_{L^2(\mu)}^2 \fstop$$
By Theorem \ref{t:Mosco}, for any $u \in W^{1,2}(X, d_X, \mu_X)$, there exists $u_n \in W^{1,2}(X, d_{n}, \mu_{X})$ such that $u_n$ converges to $u$ strongly in $L^2(\mu)$ and $\limsup_{n \to \infty} {\sf Ch}_{d_{n}, \mu_X}(u_n) \le \Ch_{d_X, \mu_X}(u)$.  By the lower semicontinuity of the entropy with respect to $L^2$-strong convergence, we conclude 
$${\rm Ent}(u^2) \le C_{1}{\sf Ch}_{d_{X}, \mu_X}(u) + D_{1} \|u\|_{L^2(\mu)}^2 \quad u \in W^{1,2}(X, d_{X}, \mu)\fstop$$

The proof for the stability of the Poincar\'e inequality is similar, so we omit the proof. 
\end{proof}
 
\begin{thm}[Talagrand inequality]\label{t:P1} 
Let $\seqn{X_n}$ be a sequence of mm-spaces with $X_n \prec X_{n+1}$ $(\forall n \in \N)$ and suppose that $\seqn{X_n}$ satisfies~${\rm T}(C_{n, T})$  with $C_T:=\lim_{n\to \infty}C_{n, T}<\infty$, Then the inverse limit $X$ satisfies $\TL(C_{T})$.
\end{thm}
\begin{proof} 
Let $\sigma \in \mathcal P_*(X)$ and set $\sigma_n:=(\pi_n)_*\sigma \in \mathcal P_*(X_n)$.
By the hypothesis that $X_n$ possesses~$\TL(C_{n, T})$, we have that 
$$W_2(\mu_{X_n}, \sigma_n)^2 \le C_{n, T} \ENT_{\mu_{X_n}}(\sigma_n) \fstop$$
We may assume that $ \ENT_{\mu_{X}}(\sigma)<\infty$, otherwise, the conclusion is trivially true. Since $\sup_{n \in \N} \ENT_{\mu_{X_n}}(\sigma_n) \le \ENT_{\mu_{X}}(\sigma)$,  the joint lower semicontinuity of $W_2$ yields 
 $$W_2(\mu_{X}, \sigma)^2 \le  \liminf_{n \to \infty} W_2(\mu_{X_n}, \sigma_n)^2 \le \sup_{n \in \N}C_{n,T} \ENT_{\mu_{X_n}}(\sigma_n) \le C_T\ENT_{\mu_{X}}(\sigma) <\infty.$$
By Lemma~\ref{l:MI}, we can upgrade the limit infimum to the limit and conclude 
$$W_2(\mu_{X}, \sigma)^2 = \lim_{n \to \infty} W_2(\mu_{X_n}, \sigma_n)^2 \le  \limsup_{n \to \infty}C_{n, T} \ENT_{\mu_{X_n}}(\sigma_n) \le C_{T} \ENT_{\mu_{X}}(\sigma) \fstop \qedhere$$
\end{proof}
\begin{thm}[Dimension-free/log-Harnack inequalities]\label{t:LH} 
Let $\seqn{X_n}$ be a sequence of  mm-spaces with $X_n \prec X_{n+1}$ satisfying $\DFH(C_{n})$ $($resp.\ $\LH(C_n)$$)$ with $C:=\lim_{n \to \infty}C_{n}<\infty$. Then the inverse limit $X$ satisfies $\DFH(C)$ $($resp.\ $\LH(C)$$)$.
\end{thm}
\begin{proof} 
We prove first the stability of the dimension-free Harnack inequality.
Fix \(\alpha\in(1,\infty)\), \(t>0\), and \(u\in L^\infty_+(X,\mu_X)\).
Set $u_n$ to be the conditional expectation:
\[
\mathcal B_n(X):=\pi_n^{-1}\mathcal B(X_n),
\qquad
u_n:=\mathbb E[u\mid \mathcal B_n(X)].
\]
Then \(u_n\in L^\infty_+(X,\mu_X)\), \(\|u_n\|_{L^\infty}\le \|u\|_{L^\infty}\), and
\(u_n\) is \(\mathcal B_n(X)\)-measurable. Hence,  there exists
\(v_n\in L^\infty_+(X_n,\mu_{X_n})\) such that
\[
u_n=v_n\circ\pi_n \qquad \mu_X\text{-a.e.}
\]
Since \(\{\mathcal B_n(X)\}_{n\in\mathbb N}\) is an increasing family of \(\sigma\)-algebras
generating \(\mathcal B(X)\), the martingale convergence theorem yields
\[
u_n\to u \quad \mu_X\text{-a.e. and in }L^p(X,\mu_X)
\quad\text{for every }1\le p<\infty.
\]
In particular,
\[
u_n\to u \quad\text{and}\quad u_n^\alpha\to u^\alpha
\qquad\text{strongly in }L^2(X,\mu_X).
\]

For \(r\in\mathbb R\), define
\[
\kappa_t(r):=
\begin{cases}
\dfrac{r}{1-e^{-2rt}}, & r\neq 0,\\[2mm]
\dfrac{1}{2t}, & r=0.
\end{cases}
\]
By the assumption that \(X_n\) satisfies \(\DFH(C_n)\), there exists a Borel set
\(\Omega_n\subset X_n\) with \(\mu_{X_n}(\Omega_n)=1\) such that, for every
\(x_n,y_n\in\Omega_n\),
\[
(H_t^n v_n)^\alpha(y_n)
\le
H_t^n(v_n^\alpha)(x_n)
\exp\Biggl(
\frac{\alpha \kappa_t(C_n)d_{X_n}(x_n,y_n)^2}{2(\alpha-1)}
\Biggr).
\]
Set
\[
\widetilde\Omega_n:=\pi_n^{-1}(\Omega_n).
\]
Then \(\mu_X(\widetilde\Omega_n)=1\). By Lemma~\ref{l:M=M}, for every
\(x,y\in\widetilde\Omega_n\),
\[
(H_t^{(n)}u_n)^\alpha(y)
\le
H_t^{(n)}(u_n^\alpha)(x)
\exp\Biggl(
\frac{\alpha \kappa_t(C_n)d_n(x,y)^2}{2(\alpha-1)}
\Biggr).
\]

Set
\[
A_n:=H_t^{(n)}u_n,
\qquad
B_n:=H_t^{(n)}(u_n^\alpha).
\]
By the \(L^2\)-contraction property of \(H_t^{(n)}\), Corollary~\ref{c:SR}, and the strong
\(L^2\)-convergence of \(u_n\) to \(u\), we have
\[
\begin{aligned}
\|A_n-H_tu\|_{L^2(\mu_X)}
&\le
\|H_t^{(n)}u_n-H_t^{(n)}u\|_{L^2(\mu_X)}
+
\|H_t^{(n)}u-H_tu\|_{L^2(\mu_X)} \\
&\le
\|u_n-u\|_{L^2(\mu_X)}
+
\|H_t^{(n)}u-H_tu\|_{L^2(\mu_X)}
\to 0.
\end{aligned}
\]
Similarly,
\[
B_n\to H_t(u^\alpha)
\qquad\text{strongly in }L^2(X,\mu_X).
\]
Thus there exists a subsequence \(\{n_k\}_{k\in\mathbb N}\) and a Borel set
\(\Omega_\infty\subset X\) with \(\mu_X(\Omega_\infty)=1\) such that
\[
A_{n_k}(z)\to H_tu(z),
\qquad
B_{n_k}(z)\to H_t(u^\alpha)(z)
\]
for every \(z\in\Omega_\infty\).

Define
\[
\Omega:=\Omega_\infty\cap\bigcap_{k=1}^\infty \widetilde\Omega_{n_k}.
\]
Then \(\Omega\) is Borel and \(\mu_X(\Omega)=1\). For every \(x,y\in\Omega\), we have
\[
(A_{n_k})^\alpha(y)
\le
B_{n_k}(x)
\exp\Biggl(
\frac{\alpha \kappa_t(C_{n_k})d_{n_k}(x,y)^2}{2(\alpha-1)}
\Biggr)
\qquad \forall k\in\mathbb N.
\]
Since \(C_{n_k}\to C\), \(\kappa_t(C_{n_k})\to \kappa_t(C)\), and
\(d_{n_k}(x,y)\uparrow d_X(x,y)\), passing to the limit gives
\[
(H_tu)^\alpha(y)
\le
H_t(u^\alpha)(x)
\exp\Biggl(
\frac{\alpha \kappa_t(C)d_X(x,y)^2}{2(\alpha-1)}
\Biggr)
\]
for every \(x,y\in\Omega\). This proves \(\DFH(C)\). The proof of the log-Harnack is similar (first prove it for $u+\e$, then let $\e \to 0$), so we omit the proof. 
\end{proof}

\subsection{Stability for $\EVI$}
\begin{thm}[EVI]\label{t:EVI} 
Let $\seqn{X_n}$ be a sequence of mm-spaces with $X_n \prec X_{n+1}$ satisfying $\EVI(K_n, \infty)$ with $K = \lim_{n \to \infty}K_n$
and $X:=\varprojlim X_n$ its inverse limit. 
Then $X$ satisfies $\EVI(K,\infty)$.
\end{thm}
\begin{proof}
We divide the proof into several steps.
We suppose that $\sigma=\rho\cdot \mu_X \in D(\ENT_{\mu_X})$ and  $\nu \in D(\ENT_{\mu_X})$ satisfy $W_{2, d_X}(\sigma, \nu)<\infty$.
Let $\{H_t^n\}$ (resp.\ $\{H_t^{(n)}\}$) denote the heat semigroup  corresponding to $\Ch_{d_{X_n}, \mu_{X_n}}$ (resp.\ $\Ch_{d_{n}, \mu_X}$). 
We set $\sigma_n:=(\pi_n)_*\sigma$ with $d\sigma_n=\rho_n d\mu_{X_n}$ and $\nu_n :=(\pi_n)_* \nu$.
Let $d\sigma^{(n)}:=\rho_n\circ \pi_n d\mu_X$.
We define the heat flow $\{H_t^n\}$ (by the same symbol as the heat semigroup), i.e., the dual action of $\{H_t^n\}$ by $H_t^n\sigma_n=(H_t^n \rho_n)\cdot \mu_{X_n}$, and analogously for $H_t^{(n)}$. 

\begin{lem}\label{l:Hn}
If $\sigma, \nu \in D(\ENT_{\mu_X})$, then
$$\Ent{\mu_X}{H_t\sigma} \le \liminf_{n \to \infty}\Ent{\mu_{X_n}}{H_t^n\sigma_n}, \quad W_{d_X, 2}(H_t\sigma, \nu) \le \liminf_{n \to \infty}W_{d_{X_n}, 2}(H_t^{n}\sigma_n, \nu_n) \fstop$$
\end{lem}
\begin{proof}
We note that $H^n_t \sigma \in D(\ENT_{\mu_{X_n}})$ whenever $\sigma = \rho\cdot\mu_{X_n} \in D(\ENT_{\mu_{X_n}})$
by the fact that $H^n_t \sigma$ is a metric gradient flow of $\ENT_{\mu_{X_n}}$, e.g.~\cite[Proposition~2.22]{AGS;D},
or by simple applications of Jensen's inequality and the invariance $\mu_{X_n}(H^n_t \rho) = \mu_{X_n}(\rho)$:
\begin{align*}
\Ent{\mu_{X_n}}{H^n_t (\rho \cdot \mu_{X_n})}
&= \int_{X_n} (H^n_t \rho)\log H^n_t \rho \diff \mu_{X_n} \\
&\le \int_{X_n} H^n_t (\rho \log \rho) \diff \mu_{X_n} = \int_{X_n} \rho \log \rho \diff \mu_{X_n}<\infty  \fstop
\end{align*}
Therefore, recalling Lemma \ref{l:MI}, 
\begin{align*}
 \sup_{n \in \N}\Ent{\mu_{X}}{H_t^{(n)}\sigma^{(n)}} \le \sup_{n \in \N}\Ent{\mu_{X}}{\sigma^{(n)}}\le \Ent{\mu_X}{\sigma}<\infty \fstop
\end{align*}
Thus, by Danford--Pettis theorem, $\{H_t^{(n)}\sigma^{(n)}\}$ is tight 
and there exist a weakly converging (non-relabelled) subsequence and its limit $\sigma'$. By Corollary~\ref{c:SR} and the $L^1$-contraction property of $H^{(n)}_t$ and $H_t$, we have $\sigma'=H_t\sigma$. The lower semicontinuity of the entropy yields 
$$\Ent{\mu_X}{H_t\sigma} \le \liminf_{n \to \infty}\Ent{\mu_X}{H_t^{(n)}\sigma^{(n)}} = \liminf_{n \to \infty}\Ent{\mu_{X_n}}{H_t^n\sigma_n} \fstop$$
Furthermore, noting that $\nu^{(n)}$ converges weakly to $\nu$ and Theorem~\ref{t:AES}, we conclude 
$$\liminf_{n \to \infty}W_{d_{X_n}, 2}(H_t^{n}\sigma_n, \nu_n)= \liminf_{n \to \infty}W_{d_n, 2}(H_t^{(n)}\sigma^{(n)}, \nu^{(n)}) \ge W_{d_X, 2}(H_t\sigma, \nu) \fstop $$
\end{proof}

{\it Resuming the proof of Theorem \ref{t:EVI}.}
We note that the inverse limit $X$ is infinitesimally Hilbertian by the hypothesis that $X_n$ satisfies $\EVI(K_n, \infty)$, which implies that $X_n$ is infinitesimally Hilbertian for every $n$ (e.g., see \cite[Theorem~5.1]{AGS;D}), and by Corollary~\ref{c:IH}, this extends to the limit $X$.
We have that $H_t\sigma \in D(\operatorname{Ent}_{\mu_X})$ whenever $\sigma=\rho\cdot\mu_X \in D(\operatorname{Ent}_{\mu_X})$ by the same argument as in the proof of Lemma~\ref{l:Hn}.

For $r\in\mathbb R$ and $t>0$, set
\[
I_r(t):=
\begin{cases}
\dfrac{1-e^{-rt}}{r}, & r\ne 0,\\[2mm]
t, & r=0.
\end{cases}
\]
We first prove the following integrated form of the EVI on the limit space:
\begin{align}
\label{e:EVI-int-limit}
\frac{1}{2}W_{d_X,2}(H_t\sigma,\nu)^2
-\frac{1}{2}e^{-Kt}W_{d_X,2}(\sigma,\nu)^2
\le
I_K(t)\bigl(\operatorname{Ent}_{\mu_X}(\nu)-\operatorname{Ent}_{\mu_X}(H_t\sigma)\bigr).
\end{align}
Indeed, set
\[
F_n(s):=\frac{1}{2}W_{d_{X_n},2}(H_s^n\sigma_n,\nu_n)^2.
\]
The EVI inequality on $X_n$ gives, for a.e. $s>0$,
\[
F_n'(s)+K_nF_n(s)
\le
\operatorname{Ent}_{\mu_{X_n}}(\nu_n)-\operatorname{Ent}_{\mu_{X_n}}(H_s^n\sigma_n).
\]
Since the entropy is non-increasing along the heat flow,
\[
\operatorname{Ent}_{\mu_{X_n}}(H_s^n\sigma_n)
\ge
\operatorname{Ent}_{\mu_{X_n}}(H_t^n\sigma_n)
\qquad 0\le s\le t.
\]
Thus, for a.e. $s\in(0,t)$,
\[
F_n'(s)+K_nF_n(s)
\le
\operatorname{Ent}_{\mu_{X_n}}(\nu_n)-\operatorname{Ent}_{\mu_{X_n}}(H_t^n\sigma_n).
\]
Multiplying by $e^{K_ns}$ and integrating from $0$ to $t$, we obtain
\begin{align}
\label{e:EVI-int-n}
\frac{1}{2}W_{d_{X_n},2}(H_t^n\sigma_n,\nu_n)^2
-\frac{1}{2}e^{-K_nt}W_{d_{X_n},2}(\sigma_n,\nu_n)^2
\le
I_{K_n}(t)\bigl(\operatorname{Ent}_{\mu_{X_n}}(\nu_n)-\operatorname{Ent}_{\mu_{X_n}}(H_t^n\sigma_n)\bigr).
\end{align}
We pass to the limit in \eqref{e:EVI-int-n}. By Lemmas~\ref{l:MI} and~\ref{l:Hn},
\[
W_{d_X,2}(H_t\sigma,\nu)
\le
\liminf_{n\to\infty}W_{d_{X_n},2}(H_t^n\sigma_n,\nu_n),
\]
\[
W_{d_{X_n},2}(\sigma_n,\nu_n)\to W_{d_X,2}(\sigma,\nu),
\qquad
\operatorname{Ent}_{\mu_{X_n}}(\nu_n)\le \operatorname{Ent}_{\mu_X}(\nu),
\]
and
\[
\operatorname{Ent}_{\mu_X}(H_t\sigma)
\le
\liminf_{n\to\infty}\operatorname{Ent}_{\mu_{X_n}}(H_t^n\sigma_n).
\]
Since $K_n\to K$, we also have $e^{-K_nt}\to e^{-Kt}$ and $I_{K_n}(t)\to I_K(t)$. Taking the limit in \eqref{e:EVI-int-n}, we obtain \eqref{e:EVI-int-limit}. In particular,
\[
W_{d_X,2}(H_t\sigma,\nu)<\infty,
\qquad
\operatorname{Ent}_{\mu_X}(H_t\sigma)<\infty.
\]

We now recover the differential EVI. Rearranging \eqref{e:EVI-int-limit}, we get
\begin{align*}
&\frac{\frac{1}{2}W_{d_X,2}(H_t\sigma,\nu)^2-\frac{1}{2}W_{d_X,2}(\sigma,\nu)^2}{t}
+\frac{1-e^{-Kt}}{t}\frac{1}{2}W_{d_X,2}(\sigma,\nu)^2 \\
&\qquad\le
\frac{I_K(t)}{t}\bigl(\operatorname{Ent}_{\mu_X}(\nu)-\operatorname{Ent}_{\mu_X}(H_t\sigma)\bigr).
\end{align*}
Letting $t\downarrow0$, and using
\[
\frac{1-e^{-Kt}}{t}\to K,
\qquad
\frac{I_K(t)}{t}\to 1,
\qquad
\operatorname{Ent}_{\mu_X}(H_t\sigma)\to \operatorname{Ent}_{\mu_X}(\sigma),
\]
we obtain
\[
\left.\frac{1}{2}\frac{d^+}{dt}\right|_{t=0}
W_{d_X,2}(H_t\sigma,\nu)^2
+\frac{K}{2}W_{d_X,2}(\sigma,\nu)^2
\le
\operatorname{Ent}_{\mu_X}(\nu)-\operatorname{Ent}_{\mu_X}(\sigma).
\]
Here the convergence of the entropy follows from the lower semicontinuity of entropy, the
$L^1$-continuity $H_t\rho\to\rho$, and the monotonicity of the entropy along the heat flow.
Finally, applying the same argument to $H_s\sigma$ in place of $\sigma$ and using the semigroup
property, we obtain the EVI inequality at every time $s>0$. Hence $X$ satisfies
$\EVI(K,\infty)$. 
\end{proof}

\begin{cor}
Suppose the same assumption as in Theorem \ref{t:EVI}. 
Then, the heat flow $\{H_t\}_{t \ge 0}$ on the inverse limit $X$ coincides with the $W_{2, d_X}$-gradient flow of ${\rm Ent}_{X, \mu}$.  
\end{cor}
\begin{proof}
This is a standard consequence of the EVI gradient flow property in~Theorem \ref{t:EVI}, see, e.g. ~\cite[Theorems~8.5 and 9.3]{AGS;C}.
\end{proof}


\begin{rem}
In extended metric measure spaces, it is not known whether $\RCD(K,\infty)$ implies $\EVI(K,\infty)$. Therefore, we proved  the stability of $\RCD$ and $\EVI$ separately. 
\end{rem}
\section{Invariance of the Cheeger energy under isomorphism} \label{sec:Inv}

Throughout this section, let
\(
X=(X,\tau_X,d_X,\mu_X)\) and \(Y=(Y,\tau_Y,d_Y,\mu_Y)
\)
be Polish extended mm-spaces, and assume that $X\cong Y$ in the sense of Definition~\ref{d:mmiso}.
Let $f:X\to Y$ and $X_0\subset X$ be a Borel map and a Borel set as in Definition~\ref{d:mmiso}, and set $Y_0:=f(X_0)$. We denote by $\mathcal B^*(X)$ the $\sigma$-algebra of universally measurable sets, i.e., $\mathcal B^*(X):=\cap_{\nu \in \mathcal P(X)} \mathcal B(X)^\nu$, where $\nu$ is a Borel probability in $X$ and $\mathcal B(X)^\nu$ is the completion of the Borel $\sigma$-algebra in terms of $\nu$.

\subsection*{1. Analytic $d$--closures and the isometric extension}

\begin{lem}\label{lem:closure}
Let $A\subset X$ be $\tau_X$-Borel and $r>0$. Set $A^{(r)}:=\{x:\ d_X(x,A)<r\}$ and
$\bar A:=\bigcap_{n\in\mathbb N}A^{(1/n)}$.
Then each $A^{(r)}$ is analytic, hence in $\Bu(X)$; therefore $\bar A$ is analytic and $\Bu(X)$--measurable.
Furthermore, $\mu_X(\bar A)=1$ if  $\mu_X(A)=1$.
\end{lem}

\begin{proof}
Take $r_k\uparrow r$. By the $\tau_X\times\tau_X$ lower semicontinuity of $d_X$,
the sub-level set~$\{d_X\le r_k\}$ is closed in $X\times X$. Thus, the set~$\{d_X<r\}=\bigcup_k\{d_X\le r_k\}$ as well as 
$S:=\{(x,a):a\in A,\ x \in X,\ d_X(x,a)<r\}$ are Borel in $X\times X$ and the projection to the first coordinate $A^{(r)}=\mathrm{proj}_X(S)$ is analytic.
Note that the countable intersection over $n$ preserves the analyticity of $\bar A$. Since $A\subset\bar A$ and $\mu_X(A)=1$, we have $\mu_X(\bar A)=1$.
\end{proof}

\begin{lem}\label{lem:extend}
Let $\bar X_0$ be the $d_X$-closure of $X_0$ and let $\bar Y_0$ be the $d_Y$-closure of $Y_0=f(X_0)$.
Then $f$ extends uniquely to an isometry $\bar f:\bar X_0\to\bar Y_0$, and 
the graph $\mathrm{Graph}(\bar f)\subset X\times Y$ is analytic; in particular $\bar f$ is $\Bu$--measurable.
Moreover $\bar f(\bar X_0)=\bar Y_0$ (hence $\bar f$ is bijective between $\bar X_0$ and $\bar Y_0$).
Furthermore, for every analytic $E\subset \bar Y_0$,
\begin{equation}\label{eq:pushF}
\mu_Y(E)=\mu_X(\bar f^{-1}(E)),
\end{equation}
where $\bar f^{-1}(E)\in\Bu(X)$.
\end{lem}

\begin{proof}
\emph{(Extension and isometry.)} For $x\in\bar X_0$, we choose $x_n\in X_0$ such that $d_X(x_n,x)\to0$.
Then $d_Y(f(x_n),f(x_m))=d_X(x_n,x_m)\le d_X(x_n,x)+d_X(x,x_m)\to0$, so $(f(x_n))_{n \ge 1}$ is Cauchy and converges in $(Y,d_Y)$ as it is complete.
Define $\bar f(x)$ as this limit. By standard arguments, the limit is  independent of the approximating sequence. Thus, we obtain the  distance preservation.

\emph{(Analytic graph.)} Let $\mathcal A:=X_0^{\mathbb N}$ (countable product with the product $\sigma$--algebra). For $n\in\mathbb N$ define
\[
E_n:=\{(x,y,(x_k)_{k\in\mathbb N})\in X\times Y\times \mathcal A:\ d_X(x,x_n)<1/n,\ d_Y(y,f(x_n))<1/n\}.
\]
Each $E_n$ is Borel  because $f$ is Borel on $X_0$ and $\{d_{X_n}<1/n\}$ is Borel in the product topology by the proof of Lemma~\ref{lem:closure}.
Thus $E:=\bigcap_{n\in\mathbb N}E_n$ is Borel in $X\times Y\times\mathcal A$. The projection $P:=\mathrm{proj}_{X\times Y}(E)$ is analytic.
It is easy to see that $(x,y)\in P$ if and only if $x\in\bar X_0$ and $y=\bar f(x)$, hence we have $P=\mathrm{Graph}(\bar f)$.

\emph{(Bijectivity onto $\bar Y_0$.)} The set $\bar X_0$ is $d_X$-closed in the complete space $(X,d_X)$, hence it is complete.
An isometry maps complete sets to complete sets, so $\bar f(\bar X_0)$ is complete in $(Y,d_Y)$ and therefore $d_Y$-closed in $Y$.
Since $\bar f(\bar X_0)$ contains $Y_0=f(X_0)$, it contains the $d_Y$-closure $\bar Y_0$.
Conversely, by construction $\bar f(\bar X_0)\subset \bar Y_0$, hence $\bar f(\bar X_0)=\bar Y_0$.

\emph{(Measure preservation.)} Since $\mu_X(\bar X_0\setminus X_0)=0$ by Lemma~\ref{lem:closure} and $\bar f=f$ on $X_0$,
\[
\mu_X(\bar f^{-1}(E))=\mu_X(X_0\cap \bar f^{-1}(E))=\mu_X(f^{-1}(E))=\mu_Y(E). \qedhere
\]
\end{proof}

%

\begin{lem}\label{lem:inverse}
Let $\bar g:=\bar f^{-1}:\bar Y_0\to\bar X_0$. Then $\bar g$ is an isometry and is $\Bu$-measurable.
Moreover, for every $\tau_X$-Borel set $E\subset X$,
\begin{equation}\label{eq:pushG}
\mu_X(E)=\mu_Y(\bar g^{-1}(E)),
\end{equation}
where $\bar g^{-1}(E)=\bar f(E\cap\bar X_0)$ is $\Bu(Y)$-measurable.
\end{lem}

\begin{proof}
We note that $\bar g$ is an isometry since it is the inverse of an isometry. As its graph is the coordinate flip of $\mathrm{Graph}(\bar f)$, it is analytic and, in particular,  $\bar g$ is $\Bu$-measurable.
For a Borel set~$E$ in $X$, $\bar g^{-1}(E)=\bar f(E\cap\bar X_0)$ is analytic.
Using \eqref{eq:pushF} with the universally measurable set $\bar f(E\cap\bar X_0)$ (and completion of $\mu_Y$),
\[
\mu_Y(\bar g^{-1}(E))=\mu_Y(\bar f(E\cap\bar X_0))=\mu_X(\bar f^{-1}(\bar f(E\cap\bar X_0)))=\mu_X(E\cap\bar X_0)=\mu_X(E)
\]
where we used $\bar f$ is injective on $\bar X_0$ and $\mu_X(X\setminus\bar X_0)=0$ by Lemma~\ref{lem:closure}.
\end{proof}

\subsection{Push-forward of test plans}
Let $\mathcal T_X$ be the family of probability measures $\pi$ on the space of continuous path $C\bigl([0,1]; (X, \tau_X)\bigl)$ in $(X, \tau_X)$ such that $\pi$ is concentrated on the space of absolutely continuous curves with $ \int_0^1 |\dot \gamma_t|^2 dt <\infty$ (we denote by $\AC([0,1]; (X, d_X))$ for the space of such curves) and satisfies 
$$\int \int_0^1 |\dot \gamma_t|^2 dt d\pi(\gamma)<\infty , \qquad (e_t)_*\pi \le C \mu_X \quad t \in [0,1] \quad \text{some $C>0$ independent of $t$},$$
where $e_t: C\bigl([0,1]; (X, \tau_X)\bigl) \to X$ is the evaluation map $\gamma \mapsto \gamma_t \in X$. We call $\pi \in \mathcal T_X$ {\it test plan} and $C$ {\it  compression constant}.
\begin{lem}\label{lem:pi-full}
Let $\pi\in\mathcal T_X$ and set $S_{X_0}:=\bigcap_{q\in\mathbb Q\cap[0,1]} e_q^{-1}(X_0)$.
Then $S_{X_0}$ is Borel in $C([0,1];X)$ and $\pi(S_{X_0})=1$.
Moreover, for every $\gamma\in S_{X_0}\cap\AC([0,1];(X,d_X))$, it holds~$\gamma([0,1])\subset\bar X_0$.
\end{lem}

\begin{proof}
Since $e_q$ is continuous, $S_{X_0}$ is Borel. Noting that  $(e_q)_*\pi\ll\mu_X$ and $\mu_X(X_0)=1$, $\pi(e_q^{-1}(X_0))=1$ for each rational number~$q$, we have $\pi(S_{X_0})=1$. If $\gamma\in\AC$, it is $d_X$-continuous since $\gamma_q\in X_0\subset\bar X_0$ for all rationals and $\bar X_0$ is $d_X$-closed, thus $\gamma_t\in\bar X_0$ for all $t$.
\end{proof}

For $\gamma\in\AC([0,1];(X,d_X))$, the metric derivative is defined as 
\begin{align} \label{e:MDD}
|\dot \gamma _t|:=\lim_{h \to 0} \frac{d_X(\gamma_{h+t}, \gamma_t)}{|h|} \quad \text{a.e.~$t$}.
\end{align}
\begin{lem}\label{lem:speed}
If $\gamma\in\AC([0,1];(X,d_X))$ takes values in $\bar X_0$, then $\eta:=\bar f\circ\gamma\in\AC([0,1];(Y,d_Y))$ and
$|\dot\eta_t|=|\dot\gamma_t|$ for a.e.\ $t$. The analogous statement holds for $\bar g$.
\end{lem}

\begin{proof}
For all $s,t$, $d_Y(\eta_s,\eta_t)=d_X(\gamma_s,\gamma_t)$, so $\eta$ is absolutely continuous if so is $\gamma$.
The metric derivative identity follows from~\eqref{e:MDD} and the isometry of $\bar f$.
The same argument applies to $\bar g$.
\end{proof}

\begin{lem}\label{lem:Phi}
Fix any curve $\bar\eta\in C([0,1];Y)$ and define $\Phi:C([0,1];X)\to C([0,1];Y)$ by
\[
\Phi(\gamma):=
\begin{cases}
\bar f\circ\gamma, & \gamma\in S_{X_0}\cap \AC([0,1];(X,d_X)),\\
\bar\eta, & \text{otherwise}.
\end{cases}
\]
Then $\Phi$ is Borel measurable and for every $\pi\in\mathcal T_X$ one has $\Phi(\gamma)=\bar f\circ\gamma$ for $\pi$-a.e.\ $\gamma$.
\end{lem}

\begin{proof}
The set $\AC=\AC([0,1];(X,d_X))$ is Borel in $C([0,1];X)$ because it is equal to~$\{\gamma:\int_0^1|\dot\gamma_t|^2dt<\infty\}$ and the $2$-energy is Borel.
Hence $S_{X_0}\cap\AC$ is Borel. For each rational number~$q$, on $S_{X_0}\cap\AC$, we have $e_q(\Phi(\gamma))=\bar f(\gamma_q)=f(\gamma_q)$ since $\gamma_q\in X_0$.
Thus $e_q\circ\Phi$ is Borel, and because the Borel $\sigma$--algebra of $C([0,1];Y)$ is generated by $\{e_q\}_{q\in\mathbb Q\cap[0,1]}$,
$\Phi$ is Borel. Finally, $\pi(S_{X_0}\cap\AC)=1$ for $\pi\in\mathcal T_X$, so the $\pi$-a.e.\ identity follows.
\end{proof}

\begin{prop}\label{prop:pushF}
If $\pi\in\mathcal T_X$ has compression constant $C$ and $\pi^f:=\Phi_*\pi$, then $\pi^f\in\mathcal T_Y$ with the same constant.
\end{prop}

\begin{proof}
By Lemma~\ref{lem:Phi} and Lemma~\ref{lem:speed}, $\pi^f$ is concentrated on $\AC([0,1];(Y,d_Y))$ and $\int \int_0^1 |\dot \eta_t|^2 dt d\pi^f(\eta)<\infty$.
For Borel $E\subset Y$ and any $t\in[0,1]$,
\[
(e_t)_*\pi^f(E)=\pi(\gamma_t\in \bar f^{-1}(E))=(e_t)_*\pi(\bar f^{-1}(E))\le C\,\mu_X(\bar f^{-1}(E))=C\,\mu_Y(E),
\]
where the inequality extends to $\Bu(X)$ and the last identity follows by~\eqref{eq:pushF}.
\end{proof}

\begin{prop}\label{prop:pushG}
Fix any curve $\bar\gamma\in C([0,1];X)$ and define $\Psi:C([0,1];Y)\to C([0,1];X)$ by
\[
\Psi(\eta):=
\begin{cases}
\bar g\circ\eta, & \eta\in S_{Y_0}\cap \AC((0,1);(Y,d_Y)),\\
\bar\gamma, & \text{otherwise}.
\end{cases}
\]
Then $\Psi$ is Borel measurable and for every $\sigma\in\mathcal T_Y$ one has $\Psi(\eta)=\bar g\circ\eta$ for $\sigma$-a.e.\ $\eta$. Furthermore, 
$\sigma^g:=\Psi_*\sigma$ is a test plan in $\mathcal T_X$ with the same compression constant as $\sigma$.
\end{prop}

\begin{proof}
The proof is similar to Proposition~\ref{lem:Phi} and~\ref{prop:pushF}.
\end{proof}

\subsection{Weak upper gradients and energy invariance}
A property is said to hold for \emph{$\mathcal T_X$-a.e.\ curve} if it holds $\pi$-a.e.\ for every $\pi\in\mathcal T_X$.
\begin{defn}[Weak upper gradient]
Let $u:X\to\mathbb R$ be $\mu_X$-measurable and let $G:X\to[0,\infty]$ be $\mu_X$-measurable.
We say that $G$ is a \emph{$\mathcal T_X$-weak upper gradient} of $u$ if for $\mathcal T_X$-a.e.\ absolutely continuous curve
$\gamma\in \AC([0,1];(X,d_X))$ one has
\begin{equation}\label{eq:wug}
\bigl|u(\gamma_1)-u(\gamma_0)\bigr|
\le \int_0^1 G(\gamma_t)\,|\dot\gamma_t|\,dt <\infty.
\end{equation}
The minimal weak upper gradient $|Du|_{w, X}$ is the element uniquely determined up to $\mu_X$-a.e.~such that $|Du|_{w, X} \le G$ $\mu_X$-a.e.~for every weak upper gradient~$G$ of $u$. 
\end{defn}
Due to~\cite[Theorem~6.2]{AGS;C}, the Cheeger energy given in Definition~\ref{d:Ch} by $L^2$-relaxation coincides with the one by the weak upper gradient approach:
\begin{align} \label{e:CHE}
\Ch_{d_X, \mu_X}(u)=\int_{X}|Du|_{w, X}^2 d\mu_X \ (\le +\infty) \qquad u \in L^2(X, \mu_X).
\end{align}
%

\begin{prop}\label{prop:wug}
Let $u:Y\to\mathbb R$ be $\mu_Y$-measurable and let $G$ be a $\mathcal T_Y$-weak upper gradient of $u$.
Then $G\circ\bar f$ is a $\mathcal T_X$-weak upper gradient of $u\circ\bar f$.
Consequently,
\[
|D(u\circ\bar f)|_{w,X}\le |Du|_{w,Y}\circ\bar f\qquad \mu_X\text{-a.e.}
\]
\end{prop}

\begin{proof}
Fix an arbitrary test plan $\pi\in\mathcal T_X$. Let $\pi^f:=\Phi_*\pi$.
By Proposition~\ref{prop:pushF} we have $\pi^f\in\mathcal T_Y$.
Moreover, by Lemma~\ref{lem:speed}, for $\pi$-a.e.\ $\gamma$, the curve $\eta=\bar f\circ\gamma$ satisfies
\begin{equation}\label{eq:speed-id}
|\dot\eta_t|=|\dot\gamma_t|\qquad\text{for a.e.\ }t\in(0,1).
\end{equation}
Since $G$ is a $\mathcal T_Y$-weak upper gradient of $u$, the inequality \eqref{eq:wug} holds for $\pi^f$-a.e.\ $\eta$:
\[
|u(\eta_1)-u(\eta_0)|
\le \int_0^1 G(\eta_t)\,|\dot\eta_t|\,dt.
\]
Writing $\eta=\bar f\circ\gamma$ and using \eqref{eq:speed-id}, we obtain that, for $\pi$-a.e.\ $\gamma$,
\[
\bigl|(u\circ\bar f)(\gamma_1)-(u\circ\bar f)(\gamma_0)\bigr|
=|u(\eta_1)-u(\eta_0)|
\le \int_0^1 G(\bar f(\gamma_t))\,|\dot\gamma_t|\,dt
=\int_0^1 (G\circ\bar f)(\gamma_t)\,|\dot\gamma_t|\,dt.
\]
Since $\pi\in\mathcal T_X$ was arbitrary, this proves that $G\circ\bar f$ is a $\mathcal T_X$-weak upper gradient of $u\circ\bar f$.
Taking $G=|Du|_{w,Y}$ and using the minimality property of $|D(u\circ\bar f)|_{w,X}$, we have
\[
|D(u\circ\bar f)|_{w,X}\le |Du|_{w,Y}\circ\bar f\qquad \mu_X\text{-a.e.},
\]
which is the sought statement. 
\end{proof}

\begin{prop}\label{prop:eq}
For every $\mu_Y$-measurable~$u$,
\[
|D(u\circ\bar f)|_{w,X}=|Du|_{w,Y}\circ\bar f\qquad \mu_X\text{-a.e.}
\]
\end{prop}

\begin{proof}
Set $v:=u\circ\bar f$.
We show that
\begin{equation}\label{eq:rev-ineq}
|Du|_{w,Y}\le |Dv|_{w,X}\circ \bar g
\qquad \mu_Y\text{-a.e.\ on }Y,
\end{equation}
where $\bar g:=\bar f^{-1}:\bar Y_0\to\bar X_0$ is the inverse isometry (defined $\mu_Y$-a.e.).
Composing \eqref{eq:rev-ineq} with $\bar f$ then yields
$|Du|_{w,Y}\circ\bar f\le |Dv|_{w,X}$ $\mu_X$-a.e.. Together with the opposite inequality from
Proposition~\ref{prop:wug}, we obtain  the desired equality.

Let $\sigma\in\mathcal T_Y$ be an arbitrary test plan on $Y$.
By Proposition~\ref{prop:pushG}, the push-forward
$\sigma^g:=\Psi_*\sigma\in\mathcal T_X$, where $\Psi(\eta)=\bar g\circ\eta$ $\sigma$-a.s.
In particular, $\sigma$ is concentrated on $d_Y$-continuous $\AC$ curves taking values in $\bar Y_0$,
so $\bar g\circ\eta$ is well-defined for $\sigma$-a.e.\ $\eta$.

Since $|Dv|_{w,X}$ is a $\mathcal T_X$-weak upper gradient of $v$, the inequality \eqref{eq:wug}
holds for $\sigma^g$-a.e.\ curve $\gamma$:
\begin{equation}\label{eq:wug-g}
|v(\gamma_1)-v(\gamma_0)|
\le \int_0^1 |Dv|_{w,X}(\gamma_t)\,|\dot\gamma_t|\,dt.
\end{equation}
Now, for $\sigma$-a.e.~$\eta$, set $\gamma:=\bar g\circ\eta$. Then,  $\gamma$ has law $\sigma^g$.
Moreover, since $\bar g$ is an isometry, Lemma~\ref{lem:speed} yields
\begin{equation}\label{eq:speedG}
|\dot\gamma_t|=|\dot\eta_t|\qquad\text{for a.e.\ }t\in(0,1),
\end{equation}
and because $v=u\circ\bar f$ and $\bar f\circ\bar g=\mathrm{Id}$ on $\bar Y_0$,
\[
v(\gamma_i)=v(\bar g(\eta_i))=u(\bar f(\bar g(\eta_i)))=u(\eta_i),\qquad i=0,1.
\]
Substituting these identities into \eqref{eq:wug-g} and using \eqref{eq:speedG}, we have  that, for $\sigma$-a.e.\ $\eta$,
\[
|u(\eta_1)-u(\eta_0)|
\le \int_0^1 |Dv|_{w,X}(\bar g(\eta_t))\,|\dot\eta_t|\,dt
= \int_0^1 (|Dv|_{w,X}\circ\bar g)(\eta_t)\,|\dot\eta_t|\,dt.
\]
Since $\sigma\in\mathcal T_Y$ was arbitrary, this shows that $|Dv|_{w,X}\circ\bar g$ is a $\mathcal T_Y$-weak upper gradient of $u$.
By the minimality property of $|Du|_{w,Y}$, we obtain \eqref{eq:rev-ineq}:
\[
|Du|_{w,Y}\le |Dv|_{w,X}\circ\bar g\qquad \mu_Y\text{-a.e.}
\]
Composing \eqref{eq:rev-ineq} with $\bar f$ and using that $\bar g\circ\bar f=\mathrm{Id}$ on $\bar X_0$
(and $\mu_X(X\setminus\bar X_0)=0$), we have 
\[
|Du|_{w,Y}\circ\bar f \le |Dv|_{w,X}\qquad \mu_X\text{-a.e.}
\]
Combined with Proposition~\ref{prop:wug} applied to $u$ (which gives the reverse inequality
$|Dv|_{w,X}\le |Du|_{w,Y}\circ\bar f$ $\mu_X$-a.e.), we conclude
\[
|D(u\circ\bar f)|_{w,X}=|Du|_{w,Y}\circ\bar f\qquad \mu_X\text{-a.e.},
\]
as claimed.
\end{proof}

\begin{thm}[Invariance of the Cheeger energy]\label{thm:energy}
Define $U:L^2(Y,\mu_Y)\to L^2(X,\mu_X)$ by $Uu:=u\circ\bar f$ (defined $\mu_X$-a.e.).
Then $U$ is an isometric isomorphism with inverse map $Vv:=v \circ \bar g$ and for every $u\in L^2(Y,\mu_Y)$,
\[
\Ch_{d_X,\mu_X}(Uu)=\Ch_{d_Y,\mu_Y}(u).
\]
\end{thm}

\begin{proof}
From \eqref{eq:pushF}, $\|Uu\|_{L^2(\mu_X)}=\|u\|_{L^2(\mu_Y)}$ and it is easy to check $U \circ V=Id_{L^2(\mu_X)}$ and $V \circ U=Id_{L^2(\mu_Y)}$ by $\bar f \circ \bar g=Id_{\bar Y_0}$ and $\bar g \circ \bar f = Id_{\bar X_0}$ and $\mu_X(\bar X_0)=\mu_Y(\bar Y_0)=1$. 
By Proposition~\ref{prop:eq} and \eqref{e:CHE}, 
\[
\Ch_{d_X,\mu_X}(Uu)=\int_X |D(u\circ\bar f)|_{w,X}^2\,d\mu_X
=\int_X (|Du|_{w,Y}^2\circ\bar f)\,d\mu_X
=\int_Y |Du|_{w,Y}^2\,d\mu_Y
=\Ch_{d_Y,\mu_Y}(u). \qedhere
\]
\end{proof}

\section{Stability of curvature bound and functional inequality under isomorphism} \label{sec:Sta}
We prove that the properties we studied in Section \ref{s:FI} are stable under the isomorphism introduced in Definition~\ref{d:mmiso}.
\begin{prop} \label{p:SI}
All the properties $\CD(K,\infty)$, $\RCD(K,\infty)$, $\EVI(K,\infty)$, \(\LS(C_{LS},D_{LS})\), \(\PI(C_P)\), \(\LH(K)\), \(H^\infty(K)\), and \(T(C_T)\)  are invariant under the isomorphism in the sense of Definition~\ref{d:mmiso}. 
\end{prop}

\begin{proof} We prove that if \(Y\) satisfies one of the listed properties, then so does \(X\). Since
isomorphism is an equivalence relation by Proposition~\ref{p:S}, the reverse implication follows
by applying the same argument to an isomorphism from \(Y\) to \(X\).

 Let
\(
f:X\to Y
\)
be an isomorphism in the sense of Definition~\ref{d:mmiso}. Thus there exists a Borel set
\(X_0\subset X\) with \(\mu_X(X_0)=1\) such that
\[
d_Y(f(x),f(x'))=d_X(x,x')
\qquad
\text{for all }x,x'\in X_0,
\]
and
\(
f_*\mu_X=\mu_Y.
\)
Set
\[
Y_0:=f(X_0).
\]
As in the proof of Proposition~\ref{p:S}, the set \(Y_0\) is Borel, \(\mu_Y(Y_0)=1\), and
\(f|_{X_0}:X_0\to Y_0\) admits a Borel inverse. We denote this inverse by
\[
g:Y_0\to X_0
\]
and extend \(g\) arbitrarily to a Borel map \(g:Y\to X\). Then, we have 
\[
g_*\mu_Y=\mu_X.
\]

If
\(\nu=\rho\mu_X\ll\mu_X\), then
\[
f_*\nu=(\rho\circ g)\mu_Y .
\]
Hence
\(
\operatorname{Ent}_{\mu_Y}(f_*\nu)
=
\operatorname{Ent}_{\mu_X}(\nu).
\)
Similarly, if \(\sigma,\nu\ll\mu_X\), then
\[
W_{2,d_Y}(f_*\sigma,f_*\nu)=W_{2,d_X}(\sigma,\nu).
\]
Indeed, if \(\gamma\in\operatorname{Cpl}(\sigma,\nu)\), then \((f\times f)_*\gamma\) is a
coupling of \(f_*\sigma\) and \(f_*\nu\), and since \(\gamma\) is concentrated on
\(X_0\times X_0\), we have 
\[
\int_{Y\times Y}d_Y(y,y')^2\,d(f\times f)_*\gamma(y,y')
=
\int_{X\times X}d_X(x,x')^2\,d\gamma(x,x').
\]
This gives
\[
W_{2,d_Y}(f_*\sigma,f_*\nu)\le W_{2,d_X}(\sigma,\nu).
\]
Applying the same argument to \(g\) gives the reverse inequality.

By Theorem~\ref{thm:energy}, the pullback operator
\(
U:L^2(Y,\mu_Y)\to L^2(X,\mu_X) \)
given by
\(Uu:=u\circ f
\)
is an isometric isomorphism and satisfies
\[
\Ch_{d_X,\mu_X}(Uu)=\Ch_{d_Y,\mu_Y}(u)
\qquad
\text{for every }u\in L^2(Y,\mu_Y).
\]
Therefore the \(L^2\)-gradient flows are intertwined:
\[
U H_t^Y=H_t^X U
\qquad
\text{for every }t\ge0.
\]
Indeed, this follows from the uniqueness of gradient flows of lower semicontinuous convex
functionals under  isometries on Hilbert spaces.

We now prove the invariance of the listed properties.

Suppose that \(Y\) satisfies \(\CD(K,\infty)\). Let
\(
\nu_0,\nu_1\in D(\operatorname{Ent}_{\mu_X})
\)
with
\(
W_{2,d_X}(\nu_0,\nu_1)<\infty.
\)
Set
\(
\eta_i:=f_*\nu_i (i=0,1).
\)
Then
\[
\eta_0,\eta_1\in D(\operatorname{Ent}_{\mu_Y}),
\qquad
W_{2,d_Y}(\eta_0,\eta_1)
=
W_{2,d_X}(\nu_0,\nu_1).
\]
By the \(\CD(K,\infty)\) condition on \(Y\), there exists a constant-speed
\(W_{2,d_Y}\)-geodesic \((\eta_t)_{t\in[0,1]}\) joining \(\eta_0\) and \(\eta_1\) such that
\[
\operatorname{Ent}_{\mu_Y}(\eta_t)
\le
(1-t)\operatorname{Ent}_{\mu_Y}(\eta_0)
+t\operatorname{Ent}_{\mu_Y}(\eta_1)
-\frac K2 t(1-t)W_{2,d_Y}(\eta_0,\eta_1)^2 .
\]
Define
\(
\nu_t:=g_*\eta_t.
\)
Using the entropy and Wasserstein identities above, we obtain
\[
W_{2,d_X}(\nu_s,\nu_t)
=
|s-t|W_{2,d_X}(\nu_0,\nu_1)
\]
and
\[
\operatorname{Ent}_{\mu_X}(\nu_t)
\le
(1-t)\operatorname{Ent}_{\mu_X}(\nu_0)
+t\operatorname{Ent}_{\mu_X}(\nu_1)
-\frac K2 t(1-t)W_{2,d_X}(\nu_0,\nu_1)^2 .
\]
Thus \(X\) satisfies \(\CD(K,\infty)\).

The invariance of \(\RCD(K,\infty)\) follows from the invariance of \(\CD(K,\infty)\) and the
invariance of infinitesimal Hilbertianity. 

We next prove the invariance of \(\EVI(K,\infty)\). Suppose that \(Y\) satisfies
\(\EVI(K,\infty)\). Let
\(
\sigma=\rho\mu_X\in D(\operatorname{Ent}_{\mu_X})
\)
and let
\(
\nu\in D(\operatorname{Ent}_{\mu_X})
\)
with
\(
W_{2,d_X}(\sigma,\nu)<\infty.
\)
Set
\[
\bar\sigma:=f_*\sigma,
\qquad
\bar\nu:=f_*\nu.
\]
Then \(\bar\sigma,\bar\nu\in D(\operatorname{Ent}_{\mu_Y})\) and
\(
W_{2,d_Y}(\bar\sigma,\bar\nu)
=
W_{2,d_X}(\sigma,\nu).
\)
Moreover, the intertwining property of the heat flow implies
\[
f_*(H_t^X\sigma)=H_t^Y\bar\sigma .
\]
Applying the \(\EVI(K,\infty)\) inequality on \(Y\) to \(\bar\sigma\) and \(\bar\nu\), and then
using the preservation of entropy and \(W_{2,d}\), gives
\[
\frac12\frac{d^+}{dt}W_{2,d_X}^2(H_t^X\sigma,\nu)
+\frac K2 W_{2,d_X}^2(H_t^X\sigma,\nu)
\le
\operatorname{Ent}_{\mu_X}(\nu)-\operatorname{Ent}_{\mu_X}(H_t^X\sigma).
\]
Hence \(X\) satisfies \(\EVI(K,\infty)\).

The invariance of the log-Sobolev and Poincar\'e inequalities follows from the preservation of
the \(L^2\)-norm, the mean, and the Cheeger energy. 

The invariance of the Talagrand inequality \(\TL(C_T)\) is immediate from the preservation of the entropy and the
Wasserstein distance.

It remains to discuss the Harnack inequalities. Suppose \(Y\) satisfies the dimension-free Harnack
inequality \(\DFH(K)\). Let \(u\in L^\infty(X,\mu_X)\) be non-negative. Since \(U\) is onto,
there exists \(v\in L^\infty(Y,\mu_Y)\) such that
\[
u=Uv=v\circ f
\qquad \mu_X\text{-a.e.}
\]
By the  intertwining,
\[
H_t^X u=U H_t^Yv .
\]
Let \(\Omega_Y\subset Y\) be a full-measure set on which the \(\DFH(K)\) 
holds. Then
\[
\Omega_X:=X_0\cap f^{-1}(\Omega_Y)
\]
has full \(\mu_X\)-measure. For \(x,x'\in\Omega_X\), applying the Harnack inequality on \(Y\) to
\(f(x),f(x')\), and using
\[
d_Y(f(x),f(x'))=d_X(x,x'),
\]
gives the desired \(\DFH(K)\) on \(X\).
The log-Harnack inequality is transported in exactly the same way, so we omit the proof. 
\end{proof}

\section{Examples} \label{ssec:WS}
\subsection{Infinite products of mm-spaces and their quotient spaces}
In this subsection, we discuss infinite product spaces, which are particular cases of inverse limits. 
Let $(Y_n, d_{Y_n}, \mu_{Y_n})$ be a sequence of isomorphic copies of an mm-space $(Y, d_Y, \mu_Y)$.  Recall that $\mu_Y$ is a probability measure. 
Define the infinite product extended mm-space by 
\begin{align*}
X=\prod_{n=1}^\infty Y_n,  \quad d_X(x, y)^2= \sum_{n=1}^\infty d_{Y_n}(x_n,  y_n)^2, \quad \mu_X = \bigotimes_{n=1}^\infty \mu_{Y_n} \comma
\end{align*}
where we equip $X$ with the product topology $\tau_X$, and $x:=\seqn{x_n}$ and $y:=\seqn{y_n}$ are points of $X$. The space $X$ is a particular case of inverse limits, and we call $X=(X, \tau_X, d_{X}, \mu_X)=Y^{\times \infty}$ {\it the infinite product extended mm-space}. 
\begin{cor} \label{c:G}
If $(Y, d_{Y}, \mu_{Y})$ satisfies either $\CD(K,\infty), \RCD(K,\infty), \EVI(K,\infty)$, then the corresponding property holds for  $(X, \tau_X, d_X, \mu_X)$ respectively. 
\end{cor}
\begin{proof}
Using the tensorisation property of $\CD(K,\infty)$ (\cite[Proposition 4.16]{St1}) $\RCD(K,\infty), \EVI(K,\infty)$ (\cite[Theorem 6.13]{AGS;D}), the proof follows by Theorem~\ref{t:mthm}.
\end{proof}

\begin{defn}[Abstract Wiener space]
Let \(H\) be a separable Hilbert space and let \(W\) be a separable Banach space. We say that
\((W,H,\mu)\) is an abstract Wiener space if \(H\) is continuously and densely embedded into
\(W\), and \(\mu\) is the centred Gaussian measure on \(W\) whose Cameron--Martin space is
\(H\). Equivalently, for every \(\ell\in W^*\), the random variable \(\ell\) is a centred Gaussian
random variable on \((W,\mu)\), and its variance is given by
\[
\int_W \ell(w)^2\,d\mu(w)
=
\|h_\ell\|_H^2,
\]
where \(h_\ell\in H\) is the unique element satisfying
\[
\ell(h)=\langle h,h_\ell\rangle_H
\qquad
\forall h\in H.
\]
\end{defn}

Let \((W,H,\mu)\) be an abstract Wiener space. We equip \(W\) with the Cameron--Martin
extended distance
\[
d_H(w_1,w_2):=
\begin{cases}
\|w_1-w_2\|_H, & w_1-w_2\in H,\\
+\infty, & \text{otherwise}.
\end{cases}
\]
Thus
\(
(W,\tau_W,d_H,\mu)
\)
is a Polish extended metric measure space, where \(\tau_W\) denotes the Banach topology on \(W\).
We also write
\[
X_1:=(\mathbb R^\infty,\tau_{\mathrm{prod}},d_{X_1},\gamma^{\otimes\infty})
\]
for the infinite product Gaussian space. Its Cameron--Martin space is
\[
H_{X_1}:=
\left\{
a=(a_n)_{n=1}^{\infty}\in \mathbb R^\infty:
\sum_{n=1}^{\infty}|a_n|^2<\infty
\right\},
\]
equipped with
\[
\|a\|_{H_{X_1}}^2:=\sum_{n=1}^{\infty}|a_n|^2,
\]
and
\[
d_{X_1}(a,b):=
\begin{cases}
\|a-b\|_{H_{X_1}}, & a-b\in H_{X_1},\\
+\infty, & \text{otherwise}.
\end{cases}
\]

\begin{prop}[Abstract Wiener spaces and the infinite product Gaussian space] \label{p:AWS-prod} 
Let \((W,H,\mu)\) be an infinite-dimensional abstract Wiener space. Then,  there exist Borel vector subspaces
\(
X_0\subset X_1\) and 
\(W_0\subset W
\)
with
\(
\gamma^{\otimes\infty}(X_0)=1\) and 
\(\mu(W_0)=1,
\)
and Borel linear inverse maps
\(
J:X_0\to W_0\) and 
\(K:W_0\to X_0
\)
such that
\[
J_*(\gamma^{\otimes\infty})=\mu
\]
and
\[
d_H(J(a),J(b))=d_{X_1}(a,b)
\qquad
\forall a,b\in X_0.
\]
In particular, 
\[
(X_1,\tau_{\mathrm{prod}},d_{X_1},\gamma^{\otimes\infty})
\cong
(W,\tau_W,d_H,\mu)
\]
in the sense of Definition~\ref{d:mmiso}, and 
\((W, \tau_W, d_H, \mu)\) satisfies
\(\mathrm{RCD}(1,\infty)\) and \(\mathrm{EVI}(1,\infty)\).
\end{prop}

\begin{proof}
Since \(H\) is a separable infinite-dimensional Hilbert space, we can choose a unitary isomorphism
\(
U:H_{X_1}\longrightarrow H.
\)
By the measurable extension theorem \cite[Chapter~2, Theorem~6]{Ba}, the unitary map
\(U\) admits a linear measurable extension
\(
J:X_1\longrightarrow W.
\)
Applying the same theorem to \(U^{-1}:H\to H_{X_1}\), we obtain a linear measurable extension
\(
K:W\longrightarrow X_1.
\)
Moreover,  these extensions preserve the corresponding Gaussian measures:
\[
J_*(\gamma^{\otimes\infty})=\mu,
\qquad
K_*\mu=\gamma^{\otimes\infty}.
\]
By the uniqueness of measurable linear extensions up to null sets, we also have
\[
K(J(a))=a
\qquad
\gamma^{\otimes\infty}\text{-a.e. }a\in X_1.
\]
By the Borel-realisation lemma~\cite[Chapter~2, Lemma~2]{Ba}, after restricting to Borel vector subspaces of full measure, we
may assume that there are Borel vector subspaces
\[
X_0\subset X_1,
\qquad
W_0\subset W
\]
with
\[
\gamma^{\otimes\infty}(X_0)=1,
\qquad
\mu(W_0)=1,
\]
such that \(J|_{X_0}:X_0\to W_0\) and \(K|_{W_0}:W_0\to X_0\) are Borel inverse linear maps.
Replacing \(X_0\) and \(W_0\) by these full-measure subspaces, we do not relabel and use the same notation
\(
J:X_0\to W_0\) and
\(K:W_0\to X_0.
\)

We claim that \(J\) preserves the Cameron--Martin extended distances on \(X_0\). Let
\(a,b\in X_0\). If \(a-b\in H_{X_1}\), then, as \(J\) extends \(U\),
\[
J(a)-J(b)=J(a-b)=U(a-b)\in H,
\]
therefore
\[
d_H(J(a),J(b))
=
\|U(a-b)\|_H
=
\|a-b\|_{H_{X_1}}
=
d_{X_1}(a,b).
\]
Conversely, suppose that
\[
d_H(J(a),J(b))<\infty.
\]
Then \(J(a)-J(b)\in H\). Since \(K\) extends \(U^{-1}\), we obtain
\[
a-b
=
K(J(a))-K(J(b))
=
K(J(a)-J(b))
=
U^{-1}(J(a)-J(b))
\in H_{X_1}.
\]
Thus \(d_H(J(a),J(b))<\infty\) if and only if \(d_{X_1}(a,b)<\infty\), and in that case,  the
two distances are equal. Hence
\[
d_H(J(a),J(b))=d_{X_1}(a,b)
\qquad
\forall a,b\in X_0.
\]

Since \(J_*(\gamma^{\otimes\infty})=\mu\), the Borel map~$J$ is measure preserving and distance preserving. Fix \(w_*\in W\), and define
\[
\widehat J(a):=
\begin{cases}
J(a), & a\in X_0,\\
w_*, & a\notin X_0.
\end{cases}
\]
Then \(\widehat J:X_1\to W\) is Borel, \((\widehat J)_*(\gamma^{\otimes\infty})=\mu\),
and \(\widehat J\) preserves the extended distances on \(X_0\).
Therefore
\[
X_1\cong W
\]
in the sense of Definition~\ref{d:mmiso}. The latter statement follows from~Corollary~\ref{c:G}, Proposition~\ref{p:SI} and the $\RCD(1,\infty)$ property for the one-dimensional Gaussian space~$(\R, d_\R, \gamma)$.
\end{proof}

\subsection{Quotient spaces}
%
%
%
%
%

Let \(Y=(Y,d_Y,\mu_Y)\) be a metric measure space. Assume that a compact group \(G\)
acts continuously on \(Y\) by measure-preserving isometries. Let
\(
X:=Y^{\times\infty}
\)
be the infinite product extended metric measure space, equipped with the product topology,
the product measure \(\mu_X:=\mu_Y^{\otimes\infty}\), and the extended distance
\[
d_X(x,y)^2:=\sum_{i=1}^\infty d_Y(x_i,y_i)^2,
\qquad
x=(x_i)_{i=1}^\infty,\quad y=(y_i)_{i=1}^\infty.
\]
Let \(G\) act diagonally on \(X\), viz.~
\[
g x := (g x_i)_{i=1}^\infty.
\]
Define
\(
Z:=G\backslash X
\)
and equip \(Z\) with the quotient extended distance
\(
d_Z([x],[y]) := \inf_{g\in G} d_X(x,gy)
\)
and the quotient measure
\(
\mu_Z := (q_\infty)_*\mu_X,
\)
where \(q_\infty:X\to G\backslash X\) is the quotient map. We equip \(Z\) with the quotient topology. Since \(G\) is compact and acts continuously on the
Polish space \(X\), the orbit space \(G\backslash X\) is Polish. Moreover, the quotient
distance \(d_Z\) is complete and lower semicontinuous with respect to the quotient topology
since \(d_X\) is complete and the action is by isometries. Indeed, the compactness of \(G\) implies that orbit representatives can be chosen along convergening
subsequences, so the lower semicontinuity and the completeness pass from \(d_X\) to the quotient distance.

For each \(n\in\mathbb N\), set
\[
Z_n:=G\backslash Y^{\times n},
\]
where \(G\) acts diagonally on \(Y^{\times n}\). Equip \(Z_n\) with the quotient distance
\[
d_{Z_n}([y],[z]) := \inf_{g\in G} d_{Y^{\times n}}(y,gz),
\]
where
\[
d_{Y^{\times n}}(y,z)^2:=\sum_{i=1}^n d_Y(y_i,z_i)^2,
\]
and with the quotient measure
\[
\mu_{Z_n}:=(q_n)_*\mu_Y^{\otimes n},
\]
where \(q_n:Y^{\times n}\to G\backslash Y^{\times n}\) is the quotient map.
Let
\[
r_n:Z_{n+1}\to Z_n
\]
be the map forgetting the last coordinate. 

\begin{thm} \label{t:CQ}
In the above setting,  \(\{(Z_n,r_n)\}_{n\in\mathbb N}\)
is an inverse system of metric measure spaces. Moreover,
\[
Z \cong \varprojlim Z_n
\]
in the sense of Definition~\ref{d:mmiso}.
In particular, if each \(Z_n\) satisfies one of the stability properties considered in Theorem~\ref{t:mthm},
then the same property holds for \(Z\) by Theorem~\ref{t:mthm} and
Proposition~\ref{p:SI}.
\end{thm}

\begin{proof}
We first check that \(\{Z_n,r_n\}_{n\in\mathbb N}\) is an inverse system. The map
\[
r_n:G\backslash Y^{\times(n+1)}\to G\backslash Y^{\times n}
\]
is defined by
\[
r_n([y_1,\dots,y_n,y_{n+1}]) := [y_1,\dots,y_n].
\]
This is well-defined because the \(G\)-action is diagonal. It is also \(1\)-Lipschitz. Indeed, for
\(y,z\in Y^{\times(n+1)}\),
\[
\begin{aligned}
d_{Z_n}(r_n([y]),r_n([z]))
&=
\inf_{g\in G}
d_{Y^{\times n}}\bigl((y_1,\dots,y_n),g(z_1,\dots,z_n)\bigr) \\
&\le
\inf_{g\in G}
d_{Y^{\times(n+1)}}(y,gz) \\
&=
d_{Z_{n+1}}([y],[z]).
\end{aligned}
\]
Moreover, \(r_n\) is measure-preserving as 
\[
r_n\circ q_{n+1}=q_n\circ \operatorname{pr}_{n,n+1},
\]
where \(\operatorname{pr}_{n,n+1}:Y^{\times(n+1)}\to Y^{\times n}\) is the projection onto the
first \(n\) coordinates, and
\[
(\operatorname{pr}_{n,n+1})_*\mu_Y^{\otimes(n+1)}=\mu_Y^{\otimes n}.
\]
Hence
\[
(r_n)_*\mu_{Z_{n+1}}=\mu_{Z_n}.
\]

Let
\(
Z_\infty:=\varprojlim Z_n
\)
be the inverse limit. We define a map
\(
\Phi:Z\to Z_\infty
\)
as follows. For \(x=(x_i)_{i=1}^\infty\in X\), set
\[
\Phi([x])
:=
\bigl([x_1],[x_1,x_2],\dots,[x_1,\dots,x_n],\dots\bigr).
\]
This is well-defined. Indeed, if \(x'=gx\) for some \(g\in G\), then
\[
[x'_1,\dots,x'_n]=[gx_1,\dots,gx_n]=[x_1,\dots,x_n]
\]
for every \(n\). The sequence \(\Phi([x])\) is compatible with the bonding maps \(r_n\). 
Thus, it belongs to \(Z_\infty\).

We next prove that \(\Phi\) is bijective.
First, suppose that
\[
\Phi([x])=\Phi([y]).
\]
Then, for every \(n\in\mathbb N\), there exists \(g_n\in G\) such that
\[
(x_1,\dots,x_n)=g_n(y_1,\dots,y_n).
\]
Since \(G\) is compact, there exists a subsequence \(\{g_{n_k}\}_{k=1}^\infty\) converging to some
\(g\in G\). Fix \(j\in\mathbb N\). For sufficiently large \(k\), we have \(n_k\ge j\). Thus, 
\[
x_j=g_{n_k}y_j.
\]
By the continuity of the \(G\)-action,
\[
x_j=gy_j.
\]
Since \(j\) was arbitrary, \(x=gy\). Therefore \([x]=[y]\), and \(\Phi\) is injective.

We now prove the surjectivity. Let
\(
z=(z_n)_{n=1}^\infty\in Z_\infty.
\)
We shall construct \(x=(x_i)_{i=1}^\infty\in X\) such that \(\Phi([x])=z\).
Choose a representative
\(
y^{(1)}_1\in Y
\)
of \(z_1\). Suppose that for some \(n\ge1\) we choose  a representative
\(
(y^{(n)}_1,\dots,y^{(n)}_n)\in Y^{\times n}
\)
of \(z_n\). Choose any representative
\(
(a_1,\dots,a_n,a_{n+1})\in Y^{\times(n+1)}
\)
of \(z_{n+1}\). Since
\(
r_n(z_{n+1})=z_n,
\)
we have
\[
[(a_1,\dots,a_n)]=[(y^{(n)}_1,\dots,y^{(n)}_n)]
\]
in \(G\backslash Y^{\times n}\). Hence there exists \(h_n\in G\) such that
\[
(h_n a_1,\dots,h_n a_n)=(y^{(n)}_1,\dots,y^{(n)}_n).
\]
Define
\[
(y^{(n+1)}_1,\dots,y^{(n+1)}_{n+1})
:=
(h_n a_1,\dots,h_n a_n,h_n a_{n+1}).
\]
Then this is a representative of \(z_{n+1}\), and its first \(n\) coordinates coincide with
\((y^{(n)}_1,\dots,y^{(n)}_n)\). By induction, we obtain compatible representatives. Define
\(
x_i:=y^{(i)}_i.
\)
Then
\[
(x_1,\dots,x_n)=(y^{(n)}_1,\dots,y^{(n)}_n)
\]
represents \(z_n\) for every \(n\). Thus \(\Phi([x])=z\), and \(\Phi\) is surjective.

We now prove that \(\Phi\) preserves the extended distances. Let \(x,y\in X\). By definition of
the inverse-limit distance on \(Z_\infty\),
\[
d_{Z_\infty}(\Phi([x]),\Phi([y]))
=
\sup_{n\in\mathbb N}
d_{Z_n}([x_1,\dots,x_n],[y_1,\dots,y_n]).
\]
Set
\[
D:=
\sup_{n\in\mathbb N}
d_{Z_n}([x_1,\dots,x_n],[y_1,\dots,y_n]).
\]
For every \(g\in G\) and every \(n\in\mathbb N\),
\[
d_{Z_n}([x_1,\dots,x_n],[y_1,\dots,y_n])
\le
d_{Y^{\times n}}\bigl((x_1,\dots,x_n),g(y_1,\dots,y_n)\bigr)
\le
d_X(x,gy).
\]
Taking the supremum over \(n\) and then the infimum over \(g\in G\), we obtain
\[
D\le d_Z([x],[y]).
\]

Conversely, if \(D=+\infty\), there is nothing to prove. Assume \(D<+\infty\). For each
\(n\in\mathbb N\), since \(G\) is compact and the function
\[
g\mapsto d_{Y^{\times n}}\bigl((x_1,\dots,x_n),g(y_1,\dots,y_n)\bigr)
\]
is continuous, there exists \(g_n\in G\) such that
\[
d_{Z_n}([x_1,\dots,x_n],[y_1,\dots,y_n])
=
d_{Y^{\times n}}\bigl((x_1,\dots,x_n),g_n(y_1,\dots,y_n)\bigr).
\]
By the compactness of \(G\),  passing to a subsequence, we may assume that
\[
g_{n_k}\to g
\]
for some \(g\in G\). Fix \(N\in\mathbb N\). For all sufficiently large \(k\), we have \(n_k\ge N\), hence
\[
\begin{aligned}
d_{Y^{\times N}}\bigl((x_1,\dots,x_N),g_{n_k}(y_1,\dots,y_N)\bigr)
&\le
d_{Y^{\times n_k}}\bigl((x_1,\dots,x_{n_k}),g_{n_k}(y_1,\dots,y_{n_k})\bigr) \\
&=
d_{Z_{n_k}}([x_1,\dots,x_{n_k}],[y_1,\dots,y_{n_k}]) \\
&\le D.
\end{aligned}
\]
Letting \(k\to\infty\) gives
\[
d_{Y^{\times N}}\bigl((x_1,\dots,x_N),g(y_1,\dots,y_N)\bigr)\le D.
\]
Taking the supremum over \(N\), we obtain
\(
d_X(x,gy)\le D.
\)
Therefore
\(
d_Z([x],[y])\le d_X(x,gy)\le D.
\)
Combining the two inequalities, we obtain
\[
d_Z([x],[y])
=
d_{Z_\infty}(\Phi([x]),\Phi([y])).
\]

It remains to check the measure-preserving property. Let
\(
\operatorname{Pr}_n:Z_\infty\to Z_n
\)
be the canonical projection. Then, for every \(n\),
\[
\operatorname{Pr}_n\circ\Phi\circ q_\infty
=
q_n\circ \operatorname{pr}_n,
\]
where \(\operatorname{pr}_n:X\to Y^{\times n}\) is the projection onto the first \(n\) coordinates.
Hence
\[
(\operatorname{Pr}_n)_*\Phi_*\mu_Z
=
(q_n)_*(\operatorname{pr}_n)_*\mu_X
=
(q_n)_*\mu_Y^{\otimes n}
=
\mu_{Z_n}.
\]
Since the inverse-limit measure on \(Z_\infty\) is uniquely characterized by its finite-dimensional
marginals \(\mu_{Z_n}\), we obtain
\[
\Phi_*\mu_Z=\mu_{Z_\infty}.
\]

Finally, \(\Phi\) is Borel. Indeed,
\(
\Phi\circ q_\infty
=
\bigl(q_n\circ \operatorname{pr}_n\bigr)_{n=1}^\infty
\)
is Borel and \(G\)-invariant. Since a continuous action of a compact group on a Polish
space has a standard Borel orbit space, the induced map \(\Phi:G\backslash X\to Z_\infty\) is Borel.

We have shown that \(\Phi\) is Borel, measure preserving, and distance preserving everywhere.
Therefore
\[
Z\cong Z_\infty= \varprojlim Z_n
\]
in the sense of Definition~\ref{d:mmiso}. The proof is complete.
\end{proof}

\begin{cor} \label{c:SER}
In the setting of the previous theorem, assume in addition that \(G\) is a compact Lie group and
that \(Y\) satisfies \(\RCD(K,\infty)\). Then
\[
Z=G\backslash Y^{\times\infty}
\]
satisfies \(\RCD(K,\infty)\) as well as \(\EVI(K,\infty)\).
\end{cor}

\begin{proof}
For each \(n\in\mathbb N\), the product space \(Y^{\times n}\), equipped with the product
distance and product measure, satisfies \(\RCD(K,\infty)\) by tensorisation. The diagonal
action of \(G\) on \(Y^{\times n}\) is by measure-preserving isometries. Hence, by the quotient
stability theorem in~\cite[Theorem 6.2]{GKMS}, the quotient
\(
Z_n:=G\backslash Y^{\times n}
\)
satisfies \(\RCD(K,\infty)\) for every \(n\in\mathbb N\).

By Theorem~\ref{t:CQ},
\(
Z\cong \varprojlim  Z_n
\)
in the sense of Definition~\ref{d:mmiso}. Since each \(Z_n\) satisfies \(\RCD(K,\infty)\),
Corollary~\ref{c:RCD} implies that \(\varprojlim  Z_n\) satisfies \(\RCD(K,\infty)\).
Finally, Proposition~\ref{p:SI} transfers the \(\RCD(K,\infty)\) property to \(Z\).

The \(\EVI(K,\infty)\) statement follows similarly. Since \(Z_n\) is an ordinary metric measure space satisfying \(\RCD(K,\infty)\), the 
equivalence between \(\RCD(K,\infty)\) and the \(\EVI_K\)-formulation of the heat flow in \cite{AGS;D} gives
\(\EVI(K,\infty)\) on \(Z_n\).
Thus Theorem~\ref{t:EVI} gives \(\EVI(K,\infty)\) for \(\varprojlim Z_n\), and
Proposition~\ref{p:SI} transfers it to \(Z\).
\end{proof}

\subsection{Path space and its quotient}
Let
\(
Y_m := (\mathbb{R}^m,d_{\mathbb{R}^m},\gamma_m),
\)
where \(d_{\mathbb{R}^m}\) is the Euclidean distance on \(\mathbb{R}^m\) and \(\gamma_m\) is the
centred Gaussian measure with covariance equal to the identity matrix. Let
\[
X_m := Y_m^{\times \infty}
\]
be the infinite product extended metric measure space.
Let
\[
W_m := \bigl(C([0,1],\mathbb{R}^m),\tau_u,d_{H_m},\mu_m\bigr),
\]
where \(\tau_u\) is the uniform topology,
\[
H_m := H^{1,2}_0([0,1];\mathbb{R}^m),
\qquad
d_{H_m}(w,v) :=
\begin{cases}
\|w-v\|_{H_m}, & w-v\in H_m,\\
\infty, & \text{otherwise},
\end{cases}
\]
and \(\mu_m\) is the law of the standard Brownian motion on \(\mathbb{R}^m\) starting at \(0\).

Let \(G\subset O(m)\) be a compact subgroup. The group~\(G\) acts diagonally on \(X_m\) by
\[
t\{x_n\}_{n=1}^\infty := \{t x_n\}_{n=1}^\infty
\qquad (t\in G),
\]
and pointwise on \(W_m\) by
\[
(tw)(s) := t(w(s))
\qquad (t\in G,\ w\in W_m,\ s\in[0,1]).
\]

Let \(p_X:X_m\to G\backslash X_m\) and \(p_W:W_m\to G\backslash W_m\) be the quotient maps.
Equip \(G\backslash X_m\) and \(G\backslash W_m\) with the quotient extended distances
\[
d_{G\backslash X_m}([x],[y]) := \inf_{t\in G} d_{X_m}(x,ty),
\qquad
d_{G\backslash W_m}([w],[v]) := \inf_{t\in G} d_{H_m}(w,tv),
\]
and the quotient measures
\[
\mu_{G\backslash X_m} := (p_X)_*(\gamma_m^{\otimes \infty}),
\qquad
\mu_{G\backslash W_m} := (p_W)_*(\mu_m).
\]
%

In the following lemma, we use the case \(m=1\) of the above notation. Thus \(X_1\) is the
one-dimensional infinite product Gaussian space and \(W_1\) is the one-dimensional classical
Wiener space. We denote by
\[
H_{X_1}:=
\left\{
a=(a_n)_{n=1}^\infty\in \mathbb R^\infty:
\sum_{n=1}^\infty |a_n|^2<\infty
\right\}
\]
the Cameron--Martin space of \(X_1\), equipped with the norm
\[
\|a\|_{H_{X_1}}^2:=\sum_{n=1}^\infty |a_n|^2 .
\]

\begin{thm} \label{t:WPE}
Under the above setting, 
\[
(G\backslash X_m,d_{G\backslash X_m},\mu_{G\backslash X_m})
\cong
(G\backslash W_m,d_{G\backslash W_m},\mu_{G\backslash W_m})
\]
 in the sense of Definition \ref{d:mmiso}. In particular,
\(G\backslash W_m\) satisfies
\(\mathrm{RCD}(1,\infty)\) and \(\mathrm{EVI}(1,\infty)\).

%
%
%
\end{thm}

\begin{proof}
We construct an isomorphism between \(G\backslash X_m\) and \(G\backslash W_m\).
We identify \(X_m\) with \(X_1^{\times m}\) by rearranging coordinates, and identify
\(W_m\) with \(W_1^{\times m}\) by taking coordinate functions.
Indeed, for
\[
x=\{x_n\}_{n=1}^{\infty}\in X_m,
\qquad
x_n=(x_n^{(1)},\ldots,x_n^{(m)})\in\mathbb R^m,
\]
we write
\[
x^{(j)}:=\{x_n^{(j)}\}_{n=1}^{\infty}\in X_1,
\qquad
j=1,\ldots,m.
\]
Similarly, for \(w\in W_m=C([0,1],\mathbb R^m)\), we write
\[
w=(w^{(1)},\ldots,w^{(m)}),
\qquad
w^{(j)}\in W_1.
\]
Under these identifications, the measures satisfy
\[
\gamma_m^{\otimes\infty}
=
(\gamma^{\otimes\infty})^{\otimes m},
\qquad
\mu_m=\mu_1^{\otimes m}.
\]

Since the classical Wiener space \((W_1,H_1,\mu_1)\) is an abstract Wiener space, 
Proposition~\ref{p:AWS-prod} applied to \((W_1,H_1,\mu_1)\) gives Borel linear subspaces
\(
X_{1,0}\subset X_1\) and \(
W_{1,0}\subset W_1
\)
with
\[
\gamma^{\otimes\infty}(X_{1,0})=1,
\qquad
\mu_1(W_{1,0})=1,
\]
and Borel linear inverse maps
\(
J_1:X_{1,0}\longrightarrow W_{1,0}\) and 
\(K_1:W_{1,0}\longrightarrow X_{1,0}
\)
such that
\[
(J_1)_*(\gamma^{\otimes\infty})=\mu_1
\]
and
\[
d_{H_1}(J_1(a),J_1(b))=d_{X_1}(a,b)
\qquad
\forall a,b\in X_{1,0}.
\]
Set
\[
X_{m,0}:=X_{1,0}^{\times m},
\qquad
W_{m,0}:=W_{1,0}^{\times m}.
\]
Then
\[
\gamma_m^{\otimes\infty}(X_{m,0})=1,
\qquad
\mu_m(W_{m,0})=1.
\]
Define
\(
I_m:X_{m,0}\longrightarrow W_{m,0}
\)
coordinatewise by
\[
I_m(x):=
\bigl(J_1(x^{(1)}),\ldots,J_1(x^{(m)})\bigr).
\]
Similarly define
\(
L_m:W_{m,0}\longrightarrow X_{m,0}
\)
by
\[
L_m(w):=
\bigl(K_1(w^{(1)}),\ldots,K_1(w^{(m)})\bigr).
\]
Then \(L_m\circ I_m=\mathrm{Id}_{X_{m,0}}\) and
\(I_m\circ L_m=\mathrm{Id}_{W_{m,0}}\). Moreover,
\[
(I_m)_*(\gamma_m^{\otimes\infty})=\mu_m.
\]

We next verify the distance identity. For \(x,y\in X_{m,0}\), under the above coordinate identification, the Cameron--Martin space
of \(X_m\) is \(H_{X_1}^{\times m}\), and hence
\[
d_{X_m}(x,y)^2
=
\sum_{j=1}^m d_{X_1}(x^{(j)},y^{(j)})^2.
\]
Similarly, since \(H_m=H_1^{\times m}\),
\[
d_{H_m}(I_m(x),I_m(y))^2
=
\sum_{j=1}^m
d_{H_1}\bigl(J_1(x^{(j)}),J_1(y^{(j)})\bigr)^2.
\]
By Proposition~\ref{p:AWS-prod},
\(
d_{H_1}\bigl(J_1(x^{(j)}),J_1(y^{(j)})\bigr)
=
d_{X_1}(x^{(j)},y^{(j)})
\)
for each \(j=1,\ldots,m\). Therefore
\[
d_{H_m}(I_m(x),I_m(y))=d_{X_m}(x,y)
\qquad
\forall x,y\in X_{m,0}.
\]

We now prove \(G\)-equivariance. Let
\(
t=(t_{ij})_{1\le i,j\le m}\in G\subset O(m).
\)
Since \(X_{1,0}\) is a vector subspace, \(X_{m,0}\) is \(G\)-invariant. For \(x\in X_{m,0}\),
\[
(tx)^{(i)}
=
\sum_{j=1}^m t_{ij}x^{(j)}\in X_{1,0}.
\]
Using the linearity of \(J_1\), we get
\[
J_1((tx)^{(i)})
=
\sum_{j=1}^m t_{ij}J_1(x^{(j)}).
\]
Hence
\(
I_m(tx)=tI_m(x)\) for 
every \(t\in G,\ x\in X_{m,0}.
\)

Let
\(
p_X:X_m\to G\backslash X_m\) and
\(p_W:W_m\to G\backslash W_m
\)
be the quotient maps. Define a Borel map
\(
\widehat I_m:X_m\longrightarrow G\backslash W_m
\)
by
\[
\widehat I_m(x):=
\begin{cases}
p_W(I_m(x)), & x\in X_{m,0},\\
[w_*], & x\notin X_{m,0},
\end{cases}
\]
where \(w_*\in W_m\) is fixed. The map \(\widehat I_m\) is Borel because \(X_{m,0}\) is Borel,
\(I_m:X_{m,0}\to W_m\) is Borel, and \(p_W\) is Borel. By the \(G\)-equivariance of \(I_m\),
the map \(\widehat I_m\) is \(G\)-invariant. Since a continuous action of a compact group on a
Polish space has a standard Borel orbit space, \(\widehat I_m\) factors through a Borel map
\(
\bar I_m:G\backslash X_m\longrightarrow G\backslash W_m
\)
such that
\[
\bar I_m\circ p_X=\widehat I_m.
\]

We check that \(\bar I_m\) is measure preserving. Since
\[
\mu_{G\backslash X_m}=(p_X)_*(\gamma_m^{\otimes\infty}),
\qquad
\mu_{G\backslash W_m}=(p_W)_*\mu_m,
\]
and since \(\gamma_m^{\otimes\infty}(X_{m,0})=1\), we have
\[
\begin{aligned}
(\bar I_m)_*\mu_{G\backslash X_m}
&=
(\bar I_m)_*(p_X)_*(\gamma_m^{\otimes\infty})\\
&=
(\widehat I_m)_*(\gamma_m^{\otimes\infty})\\
&=
(p_W)_*(I_m)_*(\gamma_m^{\otimes\infty})\\
&=
(p_W)_*\mu_m\\
&=
\mu_{G\backslash W_m}.
\end{aligned}
\]

Let
\(
A:=p_X(X_{m,0})\subset G\backslash X_m.
\)
Since \(X_{m,0}\) is Borel and \(p_X\) is Borel, the set \(A\) is analytic. Moreover,
\(
\mu_{G\backslash X_m}(A)=1.
\)
Choose a Borel set \(A_0\subset A\) with
\(
\mu_{G\backslash X_m}(A_0)=1.
\)
For \([x],[y]\in A_0\), choose representatives \(x,y\in X_{m,0}\). Then
\[
\begin{aligned}
d_{G\backslash W_m}(\bar I_m([x]),\bar I_m([y]))
&=
\inf_{t\in G}d_{H_m}(I_m(x),tI_m(y))\\
&=
\inf_{t\in G}d_{H_m}(I_m(x),I_m(ty))\\
&=
\inf_{t\in G}d_{X_m}(x,ty)\\
&=
d_{G\backslash X_m}([x],[y]).
\end{aligned}
\]
Here we used the \(G\)-equivariance of \(I_m\), the \(G\)-invariance of \(X_{m,0}\), and the
isometry property of \(I_m\) on \(X_{m,0}\). Thus \(\bar I_m\) is an isomorphism in the sense of
Definition~\ref{d:mmiso}.

Finally, since \((\mathbb R^m,d_{\mathbb R^m},\gamma_m)\) satisfies \(\RCD(1,\infty)\) and
\(G\subset O(m)\) is a compact subgroup, \(G\) is a compact Lie group. Hence
Corollary~\ref{c:SER} applies and \(G\backslash X_m\) satisfies \(\RCD(1,\infty)\) and
\(\EVI(1,\infty)\). The isomorphism just constructed, together with
Proposition~\ref{p:SI}, transfers these properties to \(G\backslash W_m\).
\end{proof}

\begin{rem} In \cite[Theorem~1.5]{FSS}, $\CD(1,\infty)$ has been proved for the Wiener space (the case~where $G=\{e\}$ is trivial), however,  even in this case, $\EVI(1,\infty)$ has not been explicitly proven in literature.
\end{rem}

\begin{rem}\normalfont 
In Shioya--Takatsu \cite[the third paragraph in~p.877]{ST}, they conjectured that any weak limit of compact homogeneous Riemannian manifolds is not mm-isomorphic (i.e., isomorphic as metric measure spaces) to the abstract Wiener space. In contrast, under the weaker notion of isomorphism in Definition~\ref{d:mmiso}, Theorem~\ref{t:WPE} identifies the infinite product Gaussian space and its compact orthogonal quotients with the corresponding classical Wiener path-space quotients. Consequently, Proposition~\ref{p:SI} transfers the synthetic lower Ricci curvature bounds and related functional inequalities to these spaces. 
\end{rem}

\subsection*{Acknowledgements}
The second author thanks Ryunosuke Ozawa for discussions related to this work.

\end{document}